\documentclass[10pt, reqno]{amsart}
\usepackage{amsmath}
\usepackage{amsthm}
\usepackage{amsfonts}
\usepackage{amssymb}
\usepackage{amscd}
\usepackage{xcolor}
\usepackage{multicol}
\usepackage{amstext,bbm}
\usepackage{graphicx, graphics}
\usepackage{bmpsize}
\usepackage{psfrag}

\theoremstyle{plain}

\theoremstyle{definition}

\renewcommand{\d}[1]{\ensuremath{\operatorname{d}\!{#1}}}

\newtheorem*{teoA}{Theorem A}
\newtheorem*{corollaryB}{Corollary B}
\newtheorem*{teoC}{Theorem C}
\newtheorem*{theo}{Theorem}

\newtheorem{defi}{Definition}[section]
\newtheorem{remark}{Remark}[section]
\newtheorem{claim}{Claim}[section]

\newtheorem{teo}{Theorem}[section]
\newtheorem{lemma}{Lemma}[section]
\newtheorem{prop}{Proposition}[section]

\newcommand{\N}{\mathbb{N}}
\newcommand{\R}{\mathbb{R}}
\newcommand{\mc}{\mathcal}

\newcommand{\Lcal}{\mathcal{L}}
\newcommand{\Fcal}{\mathcal{F}}
\newcommand{\tb}{\textbf}
\newcommand{\ti}{\textit}
\newcommand{\Jac}{\operatorname{Jac}}

\begin{document}

\address{Departamento de Matem\'atica - UFPE \\ Recife, PE, Brazil}
\title{Physical Measures for Certain Partially Hyperbolic Attractors on 3-Manifolds}
\author{Ricardo T. Bortolotti}
\email{ricardo@dmat.ufpe.br}
\thanks{The author was supported by CNPq and PRONEX-Dynamical Systems.}
\date{\today}

\maketitle

\begin{abstract}

In this work, we analyze ergodic properties of certain partially hyperbolic attractors whose central direction has a neutral behavior, the main feature is a condition of transversality between the projections of unstable leaves, projecting through the stable foliation. We prove that partial hyperbolic attractors satisfying this condition of transversality, neutrality in the central direction and regularity of the stable foliation admit a finite number of physical measures, coinciding with the ergodic u-Gibbs States, whose union of the basins has full Lebesgue measure. Moreover, we describe the construction of robustly nonhyperbolic attractors satisfying these properties.
\end{abstract}

\section{Introduction}\label{s.introduction}
In this work, we study ergodic properties for some attractors ``beyond uniform hyperbolicity". Our interest is in the existence and finiteness of physical measures for partially hyperbolic systems. The analysis of this kind of questions started with Sinai, Ruelle and Bowen in the case of hyperbolic systems \cite{Bowen,Bowen.Ruelle,Ruelle,Sinai} and was later extended for partially hyperbolic systems \cite{ABV,BV,Carvalho,Tsujii,VY}, Henon-like families \cite{Be.Ca,Be.Vi,Be.Yo}, Lorenz-like attractors \cite{Pesin,Sataev,Tucker} and many other types of systems \cite{araujo,lyubich,ls.young}. In most of the previous mentioned results, it is shown that the dynamical system admits a finite number of physical measures whose basins cover a full Lebesgue measure subset.

In \cite{Palis2}, Palis conjectured that every dynamical system can be approximated by another dynamical system having finitely many physical measures and the union of their basins has total Lebesgue measure. It is also possible to consider this Conjecture for the open set of partially hyperbolic systems. In the works of Bonatti-Viana \cite{BV} for ``mostly contracting'' and Alves-Bonatti-Viana \cite{ABV} for ``mostly expanding'' central direction, the existence and finiteness of physical measures for partially hyperbolic systems with some non-uniform contraction or expansion in the central direction is proved. However, the case when the central direction is neutral has not been well studied.

Also related to Palis' Conjecture, the existence of physical measures generic partially hyperbolic systems is not known. Therefore, it would be interesting to obtain results of genericity for generic partially hyperbolic systems.

Another relevant work is \cite{Tsujii}, where a generic approach is made in the case of surface endomorphism. Actually, in there the author proved that a generic partial hyperbolic endomorphism admits finitely many physical measures and the union of their basins has full Lebesgue measure. 

The present work is a first step to extend the analysis of Tsujii \cite{Tsujii} for the case of diffeomorphisms. More precisely, we consider partially hyperbolic attractors with central Lyapunov exponent close to zero and with a geometrical property of transversality between the projections of unstable leaves, projecting through the stable foliation, which implies the non joint-integrability of $E^{ss}\oplus E^{uu}$, and prove the existence of physical measures for this set of attractors.

We  summarize our results in the following theorem.

\begin{theo}
Given a diffeomorphism $f_0:M\rightarrow M$ of class $C^{r}$, $r\geq 2$, a three-dimensional manifold $M$ and  a partially hyperbolic attractor $\Lambda_0$ for $f_{0}$. If we assume that the attractor is dynamically coherent and the following conditions:

(H1) Transversality between unstable leaves via the stable foliation;

(H2) The central direction is neutral;

(H3') The stable foliation $\Fcal^{ss}$ is Lipschitz.

Then $f_{0}$ admits finite physical measures and the union of their basins has total Lebesgue measure in the basin of attraction of the attractor.

Moreover, if we assume that the attractor is robustly dynamically coherent and that

(H3) $\Fcal^{ss}_{f}$ varies continuously with $f$ in the $C^{1}$-topology.

Then, there exists an open set $\mathcal{U}$ containing $f_{0}$ such that the same result holds for every $f\in\mathcal{U}$.
\end{theo}

For precise definitions and statements of our results, see Section 2.

In the above theorem, the geometrical condition of transversality (H1) allow us to prove that u-Gibbs states of diffeomorphisms are sent by local stable projections into absolutely continuous measures. This step contains the technical part of our work and allow us  to prove that the ergodic u-Gibbs states are the physical measures.

In the last part of this work, we construct examples of partially hyperbolic attractors that are robustly nonhyperbolic and satisfy  conditions (H1), (H2) and (H3). 

We emphasize that our results include situations where the central Lyapunov exponent is null and, therefore, we can't use Pesin's theory.

It is natural to expect a weaker version of the property of transversality  that is generic among partially hyperbolic systems. Before proving such generic property, it is necessary to weaken the regularity assumptions for the stable foliation.

\subsection{Strategy of the proof}

The heart of the proof is an inequality similar to the Doeblin-Fortet inequality (also known as Lasota-Yorke inequality). We will work with a norm for finite measures defined in center-unstable manifolds that is similar to the $L^{2}$ norm for the densities of measures when they are absolutely continuous. Our aim is to see that every u-Gibbs state has certain regularity after  projecting locally by the stable foliation, which will imply that the ergodic u-Gibbs states are the physical measures.

The transversality condition plays an important role because the density of the u-Gibbs state is good \ti{a priori} only in the unstable direction, then, if the local stable projection of unstable discs give many directions, we will obtain some mass in many directions, which will imply that the projected measure into center-unstable leaves is absolutely continuous.

\subsection{Structure of the Work}

This work is organized as follows. In Section 2, we give the basic definitions and precise statements of our results. In Section 3, we define the boxes and norms that shall be used to state the Main Inequality and the technical Lemmas. In Section 4, we prove estimates that will be used to prove the Main Inequality. In Section 5, we use the Main Inequality to show that every ergodic u-Gibbs state has better regularity after projecting via local stable projections into center-unstable leaves,  we also show the existence and finiteness of physical measures and that the union of their basins has full Lebesgue measure.  Finally, in Section 6, we describe the construction of nonhyperbolic attractors that have central direction close to neutral and satisfies the transversality condition. We also check that they are robustly transitive and robustly nonhyperbolic.

\subsection*{Acknowledgements}
The author thanks Marcelo Viana for his orientation during the elaboration of his thesis, Martin Andersson, Jiagang Yang and Enrique Pujals for useful conversations.  The author would also like to thank IMPA for the kind hospitality during the period he stayed there as a student.

\section{Definitions and Statements}\label{s.statement}

In this Section, we give some definitions and state our main results.

\subsection{Prerequisites}

Let $M$ be a compact Riemannian manifold with a respective normalized Lebesgue measure $m$ and let  $f:M\rightarrow M$ be  a differentiable function of class $C^{r}$ with $r\geq 1$.


\begin{defi}
Given an $f$-invariant measure $\mu$, we define the \tb{basin of $\mu$} as the set $B(\mu)$ of points $x$ such that for every continuous function $\phi: M\rightarrow \R$ one has:
\begin{align*}
\lim_{n\rightarrow\infty}{\frac{1}{n}\sum_{i=0}^{n-1}{\phi(f^{i}(x))}} = \int_{M}{\phi \d\mu}
\end{align*}

We say that an $f$-invariant measure $\mu$ is a \textbf{physical measure} if the Lebesgue measure of $B(\mu)$ is positive.
\end{defi}


\begin{defi}
Given an invariant set $\Lambda$ for $f$, we say that $\Lambda$ is an \textbf{attractor} if there exists an open set $U\subset M$ such that $\overline{f(U)}\subset U$ and $\Lambda=\underset{n\geq 0}{\bigcap} {f^{n}(U)}$. The set $B(\Lambda)=\underset{j\geq 0}{\bigcup}{f^{-j}(U)}$ is called \tb{basin of attraction} of $\Lambda$, this is the set of points whose orbit accumulates on $\Lambda$.
\end{defi}

\begin{defi}
We say that an attractor $\Lambda$ is \textbf{partially hyperbolic} for $f$ if for every $x\in\Lambda$ there exist constants $\lambda_{ss}^{+}<\lambda_{c}^{-}<\lambda_{c}^{+}<\lambda_{uu}^{-}$ , $C>1$ and an $Df$-invariant splitting $T_{x}M=E^{ss}_{x}\oplus E^{c}_{x}\oplus E^{uu}_{x}$ such that $\lambda_{ss}^{+}<1$, $\lambda_{uu}^{-}>1$ and:
\begin{eqnarray*}
  & ||Df^nv|| & < {C}(\lambda_{ss}^{+})^n ||v|| \quad \quad \quad \quad v \in E^{ss}_x - \{0\} \\
 C^{-1}(\lambda_{c}^{-})^n ||v|| < & ||Df^{n} v|| & < C (\lambda_{c}^{+})^n ||v|| \quad \quad \quad \quad v \in E^c_x - \{0\} \\
 C^{-1}(\lambda_{uu}^{-})^n ||v|| < &||Df^n v|| & \quad \quad \quad \quad \quad \quad \quad \quad \quad \quad v \in E^{uu}_x - \{0\}
\end{eqnarray*}
\end{defi}

The subbundle $E^{ss}$ is called stable direction and $E^{uu}$ is called unstable direction. The distributions $E^{ss}$ and $E^{uu}$ integrates uniquely into invariant manifolds.

\begin{teo}\label{2.1}
The stable and unstable subbundles $E_{x}^{ss}$ and $E_{y}^{uu}$ integrate uniquely into laminations $\mc F^{ss}_x$ and $\mc F^{uu}_y$ for every $x\in B(\Lambda)$ and $y\in\Lambda$.
\end{teo}

A proof of Theorem \ref{2.1} can be found in \cite{araujo.melbourne} and \cite{HPS}. The existence of $\Fcal^{ss}$ can be guaranteed for every $x$ in a full neighbourhoud of the attractor (see \cite[Theorem 4.2]{araujo.melbourne}), but the unstable manifolds are only well defined for points in the attractor.

\begin{remark}
If $\Lambda$ is a partially hyperbolic attractor then $W^{uu}_{x}\subset\Lambda$ for every $x\in\Lambda$, hence $\Lambda=\underset{x\in\Lambda}{\cup}{W^{uu}(x)}$.
\end{remark}

\begin{teo}\label{bunching}
If the map $f$ is of class $C^{r}$ and satisfies the following bunching condition for a dominated splitting $E_{1}\oplus E_{2}$ and $k\geq 1$: 
\begin{equation}\label{bunching.condition}
\sup_{x\in\Lambda} ||D_{x}f_{|_{E_{1}}}|| \cdot \frac{||D_{x}f_{|_{E_{2}}}||^{k}}{m(D_{x}f_{|_{E_{2}}})}<1.
\end{equation}

Then, there exists an invariant foliation $\Fcal_{1}$ tangent to $E_{1}$ of class $C^{l}$, where $l=\min\{k,r-1\}$. Moreover, the foliation $\Fcal_{1}(g)$ tends to $\Fcal_{1}(f)$ in the $C^{l}$-topology  when $g$ tends to $f$ in the $C^r$-topology.
\end{teo}

For a proof of Theorem \ref{bunching}, see  \cite{araujo.melbourne} and \cite{HPS}. The $C^l$-continuity of the foliation  
 follows from the use of the $C^r$-Section Theorem \cite[Theorem 3.5]{HPS} in the proof of \cite[Theorem 4.12]{araujo.melbourne}.

Invariant sets satisfying (1) for $k=1$ are usually called normally dissipative (see \cite{Pujals}). In this case the condition says that the rate of contraction along $E_1$ is smaller than the ratio between the minimum and maximum singular value of the restriction $Df$  to $E_2$. If we take $k=1$ and $r=2$, from Theorem 2.2, we see that $\Fcal^{ss}$ is of class $C^1$ and the stable foliation of any $g$ $C^{2}$-close to $f$ is close to the stable foliation of $f$ in the $C^{1}$-topology, which means that $E^{ss}_x(g)$ is a $C^{1}$-subbundle and tends to $E^{ss}_x(f)$ when $g$ tends to $f$.

\begin{remark}
When $\Lambda$ is a transitive hyperbolic attractor (i.e., $E^{c}=0$), the works of Sinai-Ruelle-Bowen guarantee that there exists a unique physical measure $\mu$ whose basin $B(\mu)$ has full Lebesgue measure in the basin of attraction of $\Lambda$.
\end{remark}


Throughout the paper, the most important measures will be the so called u-Gibbs states, which are measures whose disintegration along unstable leaves corresponds to absolutely continuous measures with respect to the induced Lebesgue measure in each unstable leaf.

\begin{defi}
Given $x\in\Lambda$ , $r>0$ and a $C^1$ disk $\Sigma$ centered at $x$ with dimension $\dim (E^{ss}\oplus E^{c})$ and transversal to $\Fcal^{uu}$, we define the \tb{foliated box} $\Pi(x,\Sigma,r):=\underset{z\in\Lambda\cap\Sigma}{\bigcup}{\gamma^{uu}_{(z,r)}}$, where $\gamma^{uu}_{(z,r)}$ is the unstable neighborhood of $z$ of radius $r$.

We call a \tb{foliated chart} an application $\Phi_{x,\Sigma,r}:\Pi(x,\Sigma,r)\rightarrow I_{r}^{uu} \times (\Sigma\cap\Lambda)$ that is a homeomorphism into the image and restricted to each $\gamma^{uu}$ is a diffeomorphism into the horizontal.
\end{defi}

\begin{defi}
Let $\mu$ be an invariant finite Borel measure $\mu$, we say that $\mu$ is \ti{absolutely continuous with respect to the Lebesgue measure on unstable leaves} or a \tb{u-Gibbs state} if for every $x$, $\Sigma$ and $r>0$, the disintegration $\{ \mu_{z} \}_{z\in\Sigma\cap\Lambda}$ of the measure $(\Phi_{x,\Sigma,r})_{*}\mu$ with respect to the partition of $I^{uu} \times (\Sigma\cap\Lambda)$ into horizontal lines is formed by absolutely continuous measures with respect to the induced Lebesgue measures $m_{\gamma^{uu}_{(z,r)}}$ for $\hat{\mu}$-a.e $z$.
\end{defi}

It is well-known that u-Gibbs states always exist. Actually, if $D^{uu}$ is any compact disk contained in some unstable leaf and $m_{D^{uu}}$ is the normalized Lebesgue measure induced on $D^{uu}$ by the volume element associated to some Riemannian metric of $M$. We have:

\begin{prop}
Let $f$ be a $C^{2}$ diffeomorphism, $\Lambda$ a partially hyperbolic attractor and $D^{uu}$  an unstable disk. If we consider the measures $ \mu_{n}=\frac{1}{n}\sum_{i=0}^{n-1}{f^{j}_{*}(m_{D^{uu}})}$, then any accumulation point of $\mu_{n}$ is a u-Gibbs State.
\end{prop}

For more on u-Gibbs States one can check the book \cite{BDV}. We will refer to these measures by ``u-Gibbs''. The main properties on u-Gibbs are given in the following result:

\begin{prop}
Let $f: M \rightarrow M$ be a diffeomorphism of class $C^{r}$, $r\geq 2$, and $\Lambda$ a partially hyperbolic attractor for $f$, then

(1) The densities of a u-Gibbs with respect to the induced Lebesgue measure along unstable plaques are positive and bounded from zero and infinity.

(2) The support of every u-Gibbs is $W^{uu}$-saturated, in particular, it is contained in the attractor $\Lambda$.

(3) The set of u-Gibbs is non-empty, weak-$\ast$ compact and convex.

(4) The ergodic components of a u-Gibbs are also u-Gibbs.

(5) Every physical measure supported in $\Lambda$ is a u-Gibbs.
\end{prop}
\begin{proof}
See \cite{BDV}, Section 11.
\end{proof}

Property (5) above says that the u-Gibbs are the appropriate candidates for the physical measures. In this work, we prove that under certain conditions the ergodic u-Gibbs are in fact the physical measures.

\begin{prop}
Let $f$ be a  $C^{2}$  diffeomorphism, $\Lambda$  a partially hyperbolic attractor and $E\subset B(\Lambda)$ a measurable set  with positive Lebesgue measure. If we consider $m_{E}$ the restricted and normalized measure of $m$ to $E$, then any point of accumulation of the measures $\mu_{n}=\frac{1}{n}\sum_{j=0}^{n-1}{f^{j}_{*}(m_{E})}$ is a u-Gibbs.
\end{prop}

\subsection{The Transversality Condition}

Let $M$ be a three-dimensional compact manifold  and $\Lambda$ a partially hyperbolic attractor for $f$ with $\dim E^{ss}=\dim E^{c}=\dim E^{uu}=1$.

\begin{defi} We say that $\Lambda$ is \tb{dynamically coherent} if for every $x\in\Lambda$ there exist uniquely invariant manifolds $W^{cu}_{x}$ and $W^{cs}_{x}$ tangent to $E^{c}_{x}\oplus E^{uu}_{x}$ and $E^{ss}_{x}\oplus E^{c}_{x}$.
\end{defi}

The invariant manifolds $W^{cu}_{x}$ and $W^{cs}_{x}$ will be called as center-unstable and center-stable manifolds, respectively. When the attractor is dynamically coherent we have the invariant center manifolds $W^{c}$ given by $W^{c}_{x}=W^{cu}_{x}\cap W^{cs}_{x}$. When $z \notin \Lambda$ is in $W^{cu}_x$ we will denote $W^{cu}(z)$ by $W^{cu}_x$.

\begin{remark}
It is not known whether dynamical coherence is an open property or not. However, if the distribution $E^{c}$ is of class $C^{1}$ then it is open (see Theorems 7.1 and 7.4 in \cite{HPS}).
\end{remark}

\begin{defi}
Given a dynamically coherent attractor $\Lambda$ and two unstable curves of finite length $\gamma_{1}$ and $\gamma_{2}$, let $W^{cu}_i$ be the center-unstable submanifold of radius $R_{0}$ around the curve $\gamma_{i}$, $i=1,2$. We define the \tb{stable distance between these curves} by:
$$
d^{ss}(\gamma_{1}, \gamma_{2})
=\begin{cases}

&\underset{\gamma^{ss}}{\min}   \Big\{ 
   \begin{array}{cr}
   l(\gamma^{ss}) \text{ } | & \gamma^{ss} \text{ is  either a stable segment    joining } \gamma_{1} \text{ to } W^{cu}_{2} \\ 
   \quad & \text{ or a stable segment joining } \gamma_{2} \text{ to }      W^{cu}_{1}
   \end{array} 
  \Big\}
  \\
&\infty \text{ , if does not exist } \gamma^{ss} \text{ as above}
      \end{cases}
$$
\end{defi}

\begin{prop}
There exists a finite covering $\{U_{i}\}_{i\in I}$ of $\Lambda$ by open sets in $M$ and homeomorphisms $\psi_{i}:U_{i}\subset M\rightarrow I_{i}\times D_{i}\subset\R^{3}$, where $I_{i}\subset\R$ and $D_{i}$ is a ball contained in $\R^{2}$, such that for every $z\in\Lambda\cap D_{i}$ exists $a(z)\in I_{i}$ with $\psi_{i}(W^{cu}(z))\subset \{a(z)\}\times D_{i}$ and $\psi_{i}|_{W^{cu}(z)}$ is a diffeomorphism into the image.
\end{prop}

\begin{proof}
The center-unstable manifolds form a lamination. So for every point $x\in\Lambda$, considering a transversal section $I_{x}$, there exist a neighborhood $U_{x}$ and a homeomorphism $\psi_{x}:U_{x}\rightarrow I_{x}\times D$ such that $\psi_{x}(W^{cu}(z))\subset \{a(z)\}\times D$ for every $y\in\Lambda$ and every $z\in U_{x}\cap W^{cu}_{y}$, and $\psi_{x}|_{W^{cu}(z)}$ is a diffeomorphism into the image. By compactness of $\Lambda$, we can consider a finite sub-covering and the diffeomorphisms corresponding to this sub-covering.
\end{proof}

If the attractor is dynamically coherent then, by compactness, there exists some constant $R_{0}$ such that every center-unstable manifold has an internal radius greater than $R_{0}$. Fix $K_{1}\geq 1$ such that $K_{1}^{-1}\leq ||D(\psi_{i}|_{W^{cu}(z)})||\leq K_{1}$ for every $\psi_{i}$ as above and $R_{1}>0$ such that for every $x\in\Lambda$ the set $B^{cu}(x,R_{1})$ is contained in some $U_{i}$.

\begin{defi}
We say that two continuous curves $\gamma_{1}$ and $\gamma_{2}$ of finite length contained in a subset of $\R^{2}$ are \tb{$\theta$-transversal in neighborhoods of radius $r$ in $\R^{2}$} if:

For every $x_{1}\in\gamma_{1}$ , $x_{2}\in\gamma_{2}$ such that $d(x_{1},x_{2})<r$, there exist cones $C_{1}$ and $C_{2}$ with vertex at the points $x_{1}$ and $x_{2}$ such that $\gamma_{1}\cap B(x_{1},r)\subset C_{1}$, $\gamma_{2}\cap B(x_{2},r)\subset C_{2}$ and $\angle(v_{1},v_{2})\geq\theta$ for every tangent vectors $v_{1}$, $v_{2}$ at the points $x_{1}\in\gamma_{1}$, $x_{2}\in\gamma_{2}$ tangent to the cones $C_{1}$, $C_{2}$, respectively.
\end{defi}

When these curves are differentiable we think on the cones above as having arbitrarily small width around the tangent direction to the curve. But if the curve is not differentiable, then the cone must contain every possible tangent direction to the curve.

\begin{defi}
We say that two continuous curves $\gamma_{1}$ and $\gamma_{2}$ of finite length contained in the same center-unstable manifold $W^{cu}$ are \tb{$\theta$-transversal in neighborhoods of radius $r$} if $r<\frac{R_{1}K_{1}^{-1}}{2}$ and for every $i\in I$, every connected components $\tilde{\gamma}_{1}$ of $\psi_{i}(\gamma_{1}\cap U_i)$ and $\tilde{\gamma}_{2}$ of $\psi_{i}(\gamma_{2}\cap U_i)$ are $\theta$-transversal in neighborhoods of radius $r$ in $\R^{2}$.
\end{defi}

We will define a notion of transversality between two unstable curves them via the stable foliation.

\begin{defi} [Condition (H1)]
The \tb{Transversality Condition (H1)} holds if there exist constants $\epsilon_{0}>0$, $L>0$ and functions $\theta:(0,\epsilon_{0})\rightarrow \R^{+}$ and $r:(0,\epsilon_{0})\rightarrow\R^{+}$ such that the following is valid:

Given $\epsilon<\epsilon_0$, unstable curves $\gamma_1$ and $\gamma_2$ with length smaller than $L$ and a center-unstable manifold  $W^{cu}_3$  with $d^{ss}(\gamma_i,W^{cu}_3)<\frac{\epsilon_0}{2}$, $i=1,2$, and $d^{ss}(\gamma_1,\gamma_2)>\epsilon$, let $C$ be  an open set $C$ with product structure of $W^{ss}\times W^{cu}$, with $d^{ss}$ diameter smaller than $\epsilon_0$ and containing $\gamma_{1}$, $\gamma_{2}$, $W^{cu}_{3}$. Taking the stable projection $\pi^{ss}:C\rightarrow W^{cu}_{3}$, we ask that the curves $\pi^{ss}(\gamma_{1})$ and $\pi^{ss}(\gamma_{2})$ are $\theta(\epsilon)$-transversal in neighborhoods of radius $r(\epsilon)$.
\end{defi}

Every time we mean the Transversality Condition, it will be implicit that the manifold is three-dimensional and each subbundle is one-dimensional. 

\begin{figure}[h!]
  \centering
  \includegraphics[width=0.5\textwidth]{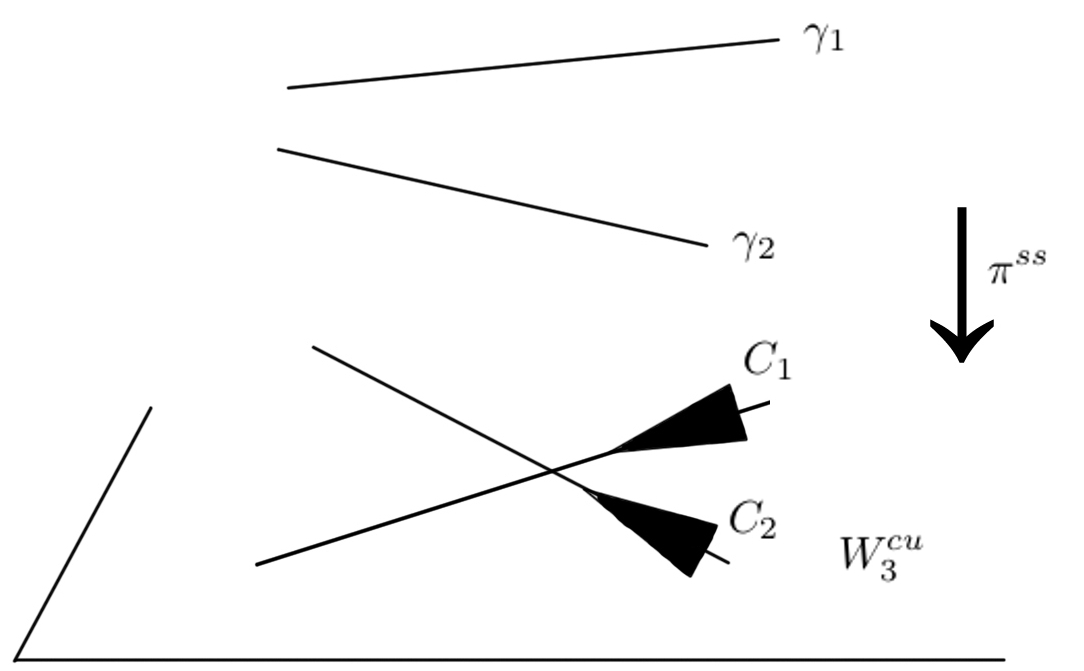}\\
\end{figure}

The transversality condition stated above is a quantitative way to say that $E^{ss}\oplus E^{uu}$ is non-integrable. It holds in several cases, for example, for contact Anosov flows \cite{C,D,L}, for algebraic automorphisms on Heisenberg nilmanifolds \cite{HHU} and for the attractor that will be constructed in Section 6. As we will see in Proposition 5.4, this condition is open for attractors where the stable foliation varies continuously in the $C^1$ topology. It is important to say that a similar condition, called uniform non-integrability (UNI), plays a fundamental role in the works of Chernov-Dolgopyat-Liverani \cite{C,D,L} for decay of correlation for contact Anosov flows.

\subsection{Main Results} 

In order to give the precise statements of our results, we will give two other main conditions, one of neutrality in the central direction and other of regularity of the stable foliation.

\begin{defi}[Condition (H2)] Let $\Lambda$ be a partially hyperbolic attractor and $\lambda_{c}^{-}$ , $\lambda_{c}^{+}$ , $\lambda_{uu}^{-}$ constants  as in Definition 2.3. The attractor $\Lambda$ has \tb{central direction  neutral} if $\lambda_{c}^{-}< 1 <\lambda_{c}^{+}$ and 
\begin{align*}
\frac{\lambda_{c}^{+}}{(\lambda_{c}^{-})^{2}\cdot\lambda_{uu}^{-}} < 1.
\end{align*}
\end{defi}

This condition of neutrality in the central direction occurs when $Df_{|_{E^{c}}}$ is close to an isometry. Note that if this condition is valid, then the center-unstable direction is volume-expanding.

\begin{defi}[Condition (H3)]

Condition (H3) holds if the stable foliation $\Fcal^{ss}_{f}$ is of class $C^{1}$ for every $f$ in a neighborhood of $f_0$ and the map $f\rightarrow \Fcal^{ss}_{f}$ is continuous at $f_{0}$ in the $C^{1}$-topology in the space of foliations defined in a fixed small neighbourhood of the attractor.
\end{defi}

This condition guarantees that the stable foliation of the attractor $\Lambda$ for $f$ are close in the $C^{1}$-topology to the ones of $\Lambda_{0}$ for $f_{0}$. It follows from Theorem 2.2 that this condition is valid when condition $\eqref{bunching.condition}$ is satisfied for $k=1$.

We now give the main results of this paper.

\begin{teoA}\label{teoA}
Given a diffeomorphism $f:M\rightarrow M$ of class $C^{r}$, $r\geq 2$,  a three-dimensional manifold $M$ and  a partially hyperbolic attractor $\Lambda$ of $f$. If we assume that the attractor is dynamically coherent and that

(H1) Transversality between unstable leaves via the stable foliation;

(H2) The central direction is neutral;

(H3') The stable foliation $\Fcal^{ss}$ is Lipschitz.

Then, $f$ admits a finite number of physical measures supported in $\Lambda$, coinciding with the ergodic u-Gibbs and the union of their basins has full measure in $B(\Lambda)$.
\end{teoA}

Theorem A will be proved in Section 5, the technical tools for the proof of this result will be developed throughout Sections 3 and 4. As a consequence of Theorem A we have the Corollary below.

\begin{corollaryB}\label{mainB}
Given a diffeomorphism $f_{0}:M\rightarrow M$ of class $C^{r}$, $r\geq 2$, a three-dimensional manifold  $M$ and  a partially hyperbolic attractor $\Lambda_{0}$ of $f_0$. If we assume that $\Lambda_{0}$ is robustly dynamically coherent and that

(H1) Transversality between unstable leaves via the stable foliation;

(H2) The central direction is neutral;

(H3) $f\rightarrow \Fcal^{ss}_{f}$ is continuous at $f_0$ in the $C^{1}$-topology.

Then, there exists an open set $\mathcal{U}$ containing $f_{0}$ such that every $f\in\mathcal{U}$ admits a finite number of physical measures supported in $\Lambda$, coinciding with the ergodic u-Gibbs and the  union of their basins has full measure in $B(\Lambda)$.
\end{corollaryB}

In Section 6, we will describe the construction of an attractor with central direction neutral satisfying the conditions of transversality and of regularity of the stable foliation. Using this construction will prove the following.

\begin{teoC}\label{mainC}
There exist $f_{0}:M^{3}\rightarrow M^{3}$ and a partial hyperbolic attractor $\Lambda_0$ that is robustly nonhyperbolic and is robustly dynamically coherent satisfying the conditions (H1), (H2) and (H3). Moreover, this attractor is robustly transitive.
\end{teoC}

The proof of Theorem C corresponds to a construction considering initially a hyperbolic solenoidal attractor satisfying the Transversality Condition (H1) and deforming the dynamics in the central direction inside a neighborhood of a periodic point in a similar way to Ma\~n\'e's example \cite{Ma}.

\section{Toolbox}\label{s.toolbox}

For the technical part of Theorem A we will make use of certain sets called boxes and certain norm that will be defined in this Section.

\subsection{The Boxes}

We will consider subsets of the manifold where  the projection through the stable foliation into one fixed center-unstable manifold is well defined and such that every unstable curve that intersects these subsets must cross them.

\begin{defi}
Given  a diffeomorphism $f:M\rightarrow M$ and a partially hyperbolic attractor $\Lambda$ that is dynamically coherent, we say that a quadruple $(C,W,\tilde{W},\pi)$ is a \tb{box} if $C$ is the image of an embedding $h:I^{uu}\times I^{a}\times I^{b}\rightarrow M$, where $I^{uu}$, $I^{a}$ and $I^{b}$ are intervals, such that:

1) The function $h_{z_{0}}$ given by $h_{z_{0}}(x,y) = h(x,y,z_{0})$ is an embedding into a surface that coincides with a center-unstable manifold if its image intersects the attractor. The set $W$ is the image of $h(I^{uu}\times I^{a} \times \{0\})$ and intersects $\Lambda$.

2) If $h(x_0,y_0,z_0)\in\Lambda$, then $\gamma(t)=h(t,y_0,z_0)$ is an unstable curve.
 
3) For every $x,y\in C$,  every connected component of $W^{cu}_{x}\cap C$ and of $W^{ss}_{y}\cap C$ intersect at most once.

4) The set $\tilde{W}$ is a connected center-unstable manifold with finite diameter containing $W$ and such that $\tilde{W}\cap W^{ss}_{y}\neq\emptyset$ for every $y\in C$. The application $\pi:C\rightarrow \tilde{W}$ sends each $y\in C$ into the point $W^{ss}_{loc}(y)\cap\tilde{W}$.
\end{defi}

The first condition above says that the set $C\cap\Lambda$ can be seen as a family of center-unstable manifolds, the second one states that if an unstable curve intersects $C$ then it crosses $C$, and the third and fourth ones guarantee that the local stable projection $\pi^{ss}: C\rightarrow \tilde{W}\supset W$ are well defined.

Since every point in the attractor admits arbitrarily small boxes containing it, it is possible to consider a finite family of boxes $\{(C_i,W_i,\tilde{W}_i,\pi_i)\}$ such that the sets $\pi_{i}^{-1}(W_{i})$ cover the attractor $\Lambda$, that is, $\Lambda\subset\underset{i}{\cup} \pi_{i}^{-1}(W_{i})$.

\subsection{The Norms}

We will define a norm that estimates a kind of regularity of the projection of measures into a fixed center-unstable manifold, it will be in terms of this norm that we will describe a criteria of absolute continuity for the stable projection of measures. This norm will be used jointly with the boxes defined in the previous section.

\begin{defi}
Let $X\subset M$ be  a submanifold of dimension 2, $m_X$ be the Lebesgue measure in $X$ induced by the Riemmanian metric, $\mu_1$ and $\mu_2$ be finite measures defined in $X$ and $r>0$ fixed. We define the bilinear form $\langle\mu_1,\mu_2\rangle_{X,r}$ by
\begin{align*}
\langle \mu_1 , \mu_2 \rangle_{X,r}=\frac{1}{r^{4}} \underset{X}{\int}{\mu_{1}(B(z,r))\mu_{2}(B(z,r)) \d m_{X}(z)}
\end{align*}

The bilinear form  $\langle\mu_1,\mu_2\rangle_{X,r}$ induces the \tb{norm}
\begin{align}
\nonumber ||\mu||_{X,r}=\sqrt{\langle \mu, \mu \rangle_{X,r}} 
\end{align}
\end{defi}
Let us now prove some facts for the norm $||\cdot||_{X,r}$.

\begin{lemma}
Given two finite families of center-unstable manifolds $\{W_{i}\}$ and $\{\tilde{W}_{i}\}$ with bounded diameter, $W_{i}\subset\tilde{W}_{i}$ for every $i$ and such that $d(W_{i}, \partial\tilde{W}_{i})>r_{0}$. Then there exists a constant $C_{0}\geq 1$ such that:
\begin{align*}
||\nu||_{\tilde{W}_{i},r_{2}} \leq C_{0} || \nu ||_{\tilde{W}_{i},r_{1}}
\end{align*}
for every $0<r_{1}\leq r_{2} \leq r_{0} \leq 1$, every $i$ and every $\nu$ supported on $W_{i}$.
\end{lemma}
\begin{proof}
Let $\tilde{C}_{0}$ be a constant independent of $r_{1}$, $r_{2}$ such that it is possible to cover every ball $B(0,r_{2})$ in the plane $\R^2$ with $\left\lceil \tilde{C}_{0} \frac{r_{2}}{r_{1}}\right\rceil^{2}$ balls of radius $r_1$. If $z_k$ are the center of these balls, then every ball $B(x,r_{2})$ is covered by the balls $B(x+z_k,r_{1})$. Since we have a finite number of center-unstable manifolds $W_i^{cu}$ that are diffeomorphic images of subsets of $\R^2$ by $\phi_i=\psi_{i}{|_{\psi_i^{-1}(W_i)}}$, we may take a uniform Lipschitz constant $K>0$ of every $\phi_i$ and $\phi_i^{-1}$ and get that the ball $B^{cu}(z,r_{2})$ is contained in $\psi_i (B(\psi_i^{-1}(z), K r_2))$, which is covered by $\left\lceil \tilde{C}_{0} \frac{K r_{2}}{K^{-1} r_{1}}\right\rceil^{2}$ sets $\psi_i B(\psi_i^{-1}(z) + z_k, K^{-1} r_1)$, and each one of them is covered by a ball $B^{cu}(\psi_i(\psi_i^{-1}(z)+z_k), r_1)$. So there exists a constant $C_{0}$ independent of $r_{1}$, $r_{2}$, $x$ and $i$ such that it is possible to cover every $B^{cu}(x,r_{2})$ in $W^{cu}$ with $\left\lceil C_{0} \frac{r_{2}}{r_{1}}\right\rceil^{2}$ disks $B^{cu}(w_{k},r_{1})$.

We have the following estimate:
\begin{align}
\nonumber ||\mu||^{2}_{\tilde{W}_{i}, r_{2}}&=\frac{1}{r_{2}^{4}}\int_{W_{i}}\mu(B(z,r_{2}))^2 \d m^{cu}(z)\\
\nonumber &\leq \frac{1}{r_{2}^{4}}\int_{\tilde{W}_{i}} \Big(\sum_{k} \mu(B(w_{k},r_{1})\Big)^{2}) \d z\\
\nonumber &\leq \frac{1}{r_{2}^{4}}C_{0}^2\left(\frac{r_{2}^2}{r_{1}^2}\right)^{2} \sum_{k}\int_{\tilde{W}_{i}}\mu(B(w_{k},r_{1}))^{2} \d w_{k} \\
\nonumber &\leq C_{0}^{2} ||\mu||_{\tilde{W}_{i}, r_{1}}^{2}
\end{align}
\end{proof}

\begin{lemma}
Let $\{\nu_{n}^{1}\}_n$ and $\nu^{1}$ be finite measures defined in a center-unstable manifold $W$ such that $\nu_{n}^{1}\overset{*}{\rightarrow}\nu_{\infty}^{1}$ and $r>0$ fixed, then:
\begin{align*}
\lim_{n\rightarrow\infty}||\nu_{n}^{1}||_{W,r}= ||\nu^{1}||_{W,r}
\end{align*}
Moreover, if we consider finite measures $\{\nu_{n}^{2}\}_n$ and $\nu^{2}$  also defined in $W$ such that $\nu_{n}^{2}\overset{*}{\rightarrow}\nu_{\infty}^{2}$, then:
\begin{align*}
\lim_{n\rightarrow\infty}\langle\nu_{n}^{1},\nu_{n}^{2}\rangle_{W,r}= \langle\nu^{1},\nu^{2}\rangle_{W,r}
\end{align*}
\end{lemma}

\begin{proof}
Note that $\nu^{1}(\partial B^{cu}(z,r))=0$ for $m^{cu}$-ae z, because:
\begin{align*}
{\int}\nu^{1}(\partial B^{cu}(z,r) \d m^{cu}(z) &= \underset{d(z,w)=r}{\int\int}\d\nu^{1}(w) dm^{cu}(z)\\
&= \int m^{cu}(\partial B^{cu}(w,r)) \d\nu^{1}(w)\\
&=0
\end{align*}

Taking $J_{r}\nu (x)=\frac{\nu(B(x,r))}{r^{2}}$, then $J_{r}\nu^{1}_{n}\rightarrow J_{r}\nu^{1}_{\infty}$ and since $||J_{r}\nu^{1}_{n}||_{L^{2}}$ is uniformly bounded, the lemma follows by the theorem of dominated convergence.

The proof that $\lim_{n\rightarrow\infty}\langle\nu_{n}^{1},\nu_{n}^{2}\rangle_{W,r}= \langle\nu_{\infty}^{1},\nu_{\infty}^{2}\rangle_{W,r}$ is analogous.
\end{proof}

Let us now consider a norm measuring the total mass of a measure.

\begin{defi}
For finite measures defined in the manifold $M$, we will consider the \tb{norm} $|\cdot|$ as: $|\mu|=\mu(M)$
\end{defi}

\begin{defi}
Given a finite family of applications $\pi_{i}:C_{i}\subset M \rightarrow \tilde{W}_{i}\supset W_{i}$, $i=1,\cdots, s_0$ and a measure $\mu$ in $M$, we consider the families $\mathcal{W}=\{ W_1, \cdots, W_{s_{0}} \}$ and $\tilde{\mathcal{W}}=\{ \tilde{W}_1, \cdots, \tilde{W}_{s_{0}} \}$ and define
\begin{align*}
|||\mu|||_{\mathcal{W},r}:=\underset{i=1,\cdots, s_0}{\max} \{|| (\pi_{i})_{*}\mu||_{W_{i},r}\}\\
|||\mu|||_{\tilde{\mathcal{W}},r}:=\underset{i=1,\cdots, s_0}{\max} \{|| (\pi_{i})_{*}\mu||_{\tilde{W}_{i},r}\}
\end{align*}
\end{defi}

\begin{remark}
Given a finite family of boxes $\{ (C_{i},W_{i},\tilde{W}_{i},\pi_{i}) \}$, $i=1,\cdots, s_{0}$, with the property that $\{\pi_{i}^{-1}(W_{i})\}$ covers $\Lambda$, for measures $\mu$ defined in $\Lambda$ we have the following equivalence for $|||\mu|||_{\{ W_{i} \},r}$ and $|||\mu|||_{\{ \tilde{W}_{i} \},r}$:
\begin{align*}
|||\mu|||_{\{ \mathcal{W} \},r} \leq |||\mu|||_{\{ \mathcal{\tilde{W}} \},r} \leq s_{0} |||\mu|||_{\{ \mathcal{W} \},r}
\end{align*}
\end{remark}

When the boxes and the corresponding center-unstable manifolds are implicit, we will simply denote the norms by $||\cdot ||_{r}$ and $|||\cdot |||_{r}$. 

Given a center-unstable manifold $W$, we denote by $m^{cu}_{W}$ the Lebesgue measure in this submanifold given by the Riemannian metric. When the set $W$ is implicit we will denote it by $m^{cu}$. Notice that we are considering dynamically coherent attractors with $\dim (E^{c}\oplus E^{uu})=2$. For absolute continuous measures defined in a center-unstable surface, the norm $||\cdot ||_{r}$ is related to the $L^{2}$ norm of the density, this relation can be seen in the following criteria of absolute continuity for measures defined in a center-unstable manifold.

\begin{lemma}\label{lemma.abs.cont}
Given a center-unstable manifold $W$, there exists a constant $I>0$ such that every finite  Borel measure $\nu$ defined in $W$ satisfying
\begin{align*}
\underset{r\rightarrow 0^{+}}{\liminf } { ||\nu||_{W,r}} \leq L,
\end{align*}
for some $L>0$, is absolutely continuous with respect to $m^{cu}$ and
\begin{align*}
\left|\left|\frac{d\nu}{d m^{cu}}\right|\right|_{L^{2}(W,m^{cu})}\leq IL.
\end{align*}
\end{lemma}

\begin{proof}
First we observe that there exist constants $C_{1}, C_{2} > 0$ such that $C_{1}\leq \frac{m^{cu}(B(x,r))}{r^{2}} \leq C_{2}$, which is valid because the center-unstable manifolds form a lamination and the areas of balls with small radius depends only on the first derivative of the restriction of these charts to the horizontal planes.

Define $J_{r_{n}}\nu(x)=\frac{\nu(B^{cu}_{W}(x,r_{n}))}{m^{cu}(B^{cu}(x,r_{n}))}$. By hypothesis, we can consider a sequence $r_{n}\rightarrow 0^{+}$ such that $||J_{r_{n}}\nu||_{L^{2}(W,m^{cu})}$ is uniformly bounded by $C_{1}^{-1}L$. Taking a subsequence we suppose that $J_{r_{n}}\nu\rightarrow\psi \in L^{2}$ weakly, so
\begin{align*}
\langle f,\psi \rangle_{L^{2}} = \underset{n\rightarrow\infty}{\lim}\underset{W}{\int} f\cdot J_{r_{n}}\nu\text{ } \d m^{cu} = \underset{W}{\int} f \d\nu
\end{align*}
for every continuous function $f$. This implies that $\nu=\psi m^{cu}_{W}$, and that for $I=C_{1}^{-1}$:
\begin{align*}
\left|\left|\frac{d\nu}{dm^{cu}}\right|\right|_{L^{2}} = \lim_{n\rightarrow \infty} ||J_{r_{n}}\nu||_{L^{2}} \leq IL
\end{align*}

\end{proof}

\section{The Main Inequality}\label{s.mainineq}

This Section is dedicated to prove the Main Inequality of this work (Proposition 4.1 below), which will be used to estimate the norm $||(\pi_{i})_{*}\mu||_{W_{i},r}$ for small parameters $r$.

Given $f$ and $\Lambda$ satisfying the assumptions of Theorem A, consider a finite family of boxes $\{(C_{i},W_{i},\tilde{W}_{i},\pi_{i})\}$, $i=1,\cdots,s_0$, such that $\Lambda\subset\underset{i}{\cup} \pi_{i}^{-1}(W_{i})$ and the $d^{ss}$-diameter of each box is smaller than the $\epsilon_{0}$ given in the Transversality Condition (H1), $r_{0}$ small such that $d^{cu}(\pi_{i}(C_{i}),\partial \tilde{W}_{i})>r_{0}$ for every $i$, fixed constants $a_{1}$ and $a_{2}$ such that for every unstable curve $\gamma^{uu}$ crossing $C_{i}$ we have $a_{1} \leq l(\gamma^{uu}) \leq a_{2}$ and an upper bound $L\geq 1$  for the Lipschitz constant of every $\pi_{i}$. In what follows, we will suppose that this family is fixed once for all. 

The Main Inequality of this work is the following.

\begin{prop}[Main Inequality]
There exist constants $B>0$ and $\sigma>1$ such that for every $n\in\mathbb{N}$, there exist $D_{n}>0$, $r_{n}>0$ and $c_{n}>1$ such that for every ergodic u-Gibbs $\mu$ and every $r\leq r_{n}$, we have:
\begin{align}
\nonumber \left|\left|\left| f^{n}_{*}\mu \right|\right|\right|^{2}_{r} \leq \frac{B}{\sigma^{n}} \left|\left|\left|\mu \right|\right|\right|^{2}_{c_{n}r} + D_{n} \left|\mu\right|^{2} 
\end{align}
\end{prop}

This type of inequality is often used in the study of the regularity of invariant measures in ergodic theory. There are two norms, one measuring the size and other measuring the regularity of the measure. These inequalities allow us to estimate the regularity of the fixed points. This kind of inequality is due to Doeblin-Fortet, and it is also known as ``Lasota-Yorke type inequality''.

The proof of Proposition 4.1 is the main objective of this Section.

\subsection{Approximating u-Gibbs inside the Boxes}

To apply our arguments, it will be useful to approximate the restriction of a u-Gibbs to one box by a finite combination of measures supported on unstable curves that crosses the box.

\begin{lemma}
There exist $C_{1}>0$ and $\alpha_{1}\in(0,1)$ such that the restriction of every ergodic u-Gibbs to a box $C_{i}$ is equals to the limit $\displaystyle \lim_{n\rightarrow\infty} \mu_{n}$ with $\mu_{n}={\sum_{j=0}^{r_{0}}{\rho_{j}^{n}m_{\gamma_{j}^{n}}}}$, where $\log\rho_{j}^{n}$ is $(C_{1},\alpha_{1})$-Holder and $\gamma_{j}^{n}$ is an unstable curve contained in $C_{i}$ that crosses $C_{i}$.
\end{lemma}
\begin{proof}
For every ergodic u-Gibbs $\mu$, there exists some unstable curve $\gamma^{uu}$ such that $\mu$ is written as a limit of $\mu_{n}:=\frac{1}{n}\sum_{j=0}^{n-1}f^{j}_{*}(m_{|_{\gamma^{uu}}})$. Given $C_{1}$ and $\alpha_{1}$, let $\Lcal_{C_{1},\alpha_{1}}$ be the closed space of measures whose logarithm of the density of the conditional measure to the unstable curves are of class $(C_{1},\alpha_{1})$-Holder, one can easily see that there exist $C_{1} , \alpha_{1}$ and $n_{0}\in \N$ such that the space $\Lcal_{C_{1},\alpha_{1}}$ is invariant under $f^{n}_{*}$ for every $n\in\N$.

Consider the measures $\nu_{n}:=\Big(\frac{1}{n}\sum_{j=0}^{n-1}f^{j}_{*}(m_{|_{\gamma^{uu}}})\Big)_{|_{C_{i}}}\rightarrow \mu_{|_{C_{i}}}$. For every $j\in\N$, let $\tilde{\gamma}^{uu}_{j}\subset f^{j}(\gamma^{uu})\cap C_{i}$ be the curve obtained removing the connected components of $f^{j}(\gamma^{uu})\cap C_{i}$ that do not cross the box $C_{i}$, and consider the measure $\tilde{\nu}_{n}:=\Big(\frac{1}{n}\sum_{j=0}^{n-1}f^{j}_{*}(m_{|_{f^{-j} \tilde{\gamma}^{uu}_{j} }})\Big)_{|_{C_{i}}}$. From the invariance of $\Lcal_{C_{1},\alpha_{1}}$, it follows that $\tilde{\nu}_{n} \in \Lcal_{C_{1},\alpha_{1}}$ for every $n\in \N$. Note that $\nu_{n}-\tilde{\nu_{n}}\overset{n\rightarrow\infty}{\rightarrow} 0$ because their difference is a measure supported in the connected components that were removed, but these connected components are at most $2n$, they have bounded length and the density of their conditional measure to these unstable curves converges to zero when $n\rightarrow\infty$. So $\displaystyle \lim_{n}\tilde{\nu}_{n}=\lim\nu_{n}=\mu_{|_{C_{i}}}$, thus the sequence $\{\tilde{\nu}_{n}\}_{n}$ satisfies the conditions that we want.
\end{proof}

\subsection{Comparing Sizes of Cylinders}

Denote by $B^{cu}(x,r)$ the center-unstable ball of radius $r$ centered at $x$. The next lemma estimates, for u-Gibbs measures, the measure of the sets $\pi_{i}^{-1}B^{cu}(x,r)$, which we will look as cylinders.

Given $n\in \N$, let $\{R_{i,j,m}\}_{m=1,\cdots,m_{0}(i,j,n)}$ be the connected components of $f^{n}(C_{i})\cap C_{j}$, the sets $C_{i,j,m}:= f^{-n}R_{i,j,m}\subset C_{i}$ and the maps $f_{i,j,m}:=f^{n}_{|_{C_{i,j,m}}}$. When the choices of $C_{i}, C_{j}$ and $n$ are implicit we will write $f_{m}$ instead of $f_{i,j,m}$.

\begin{lemma}
There exists $B_1>0$ such that for every $n$, $C_{i}$ and $C_{j}$ there exists $\tilde{r}_{n}$ such that, taking $\delta=10C^{-1}L^{2}(\lambda_{c}^{-})^{-n}r$ and $\sigma_{n,r}(x)=(\lambda_{c}^{-})^{n}\underset{w\in C(x)}{\min}{||Df^{n}_{|_{E^{uu}_{w}}}||}$ with $C(x)=f_{m}^{-1}\pi_{j}^{-1}B^{cu}(\pi_{j}f^{n}(x),r)$, we have for every ergodic u-Gibbs $\mu$, every $x\in W_{i}$, every $m$ and any $r<\tilde{r}_{n}$:
\begin{align*}
{\mu\big(f_{m}^{-1}\pi_{j}^{-1}B^{cu}(\pi_{j} f^{n}(x),r)\big)} \leq \left(\frac{r}{\delta}\right)^{2} \frac{B_{1}}{\sigma_{n,r}(x)} \cdot  {\mu\big(\pi_{i}^{-1}B^{cu}(x,\delta)\big)}
\end{align*}
\end{lemma}

\begin{proof}
First, let us prove this lemma for measures $\mu=m_{\gamma}$, where $\gamma$ is an unstable curve contained in $C_{i}$ that crosses $C_{i}$. If $f^{n}(x)\notin C_{j}$ then the left-hand side of the inequality is zero, so we can suppose that $f^{n}(x)\in C_{j}$. Consider $z\in W^{ss}(x)$ such that $\gamma\subset W^{cu}(z)$ and define the set $V(z)=W^{cu}(z)\cap f_{m}^{-1}\pi_{j}^{-1}B^{cu}(\pi_{j} f^{n}(x),r)$, let us also consider the intervals $I=\gamma^{-1}(V(z))$ and $J$ $=\gamma^{-1}(B^{cu}(z,L^{-1}\delta))$. We note that $V(z)\subset B^{cu}(z,L^{-1}\delta)\subset W^{cu}(z)\cap \pi_{i}^{-1}B^{cu}(x,\delta)$.

\begin{figure}[h!]
  \centering
  \includegraphics[width=0.8\textwidth]{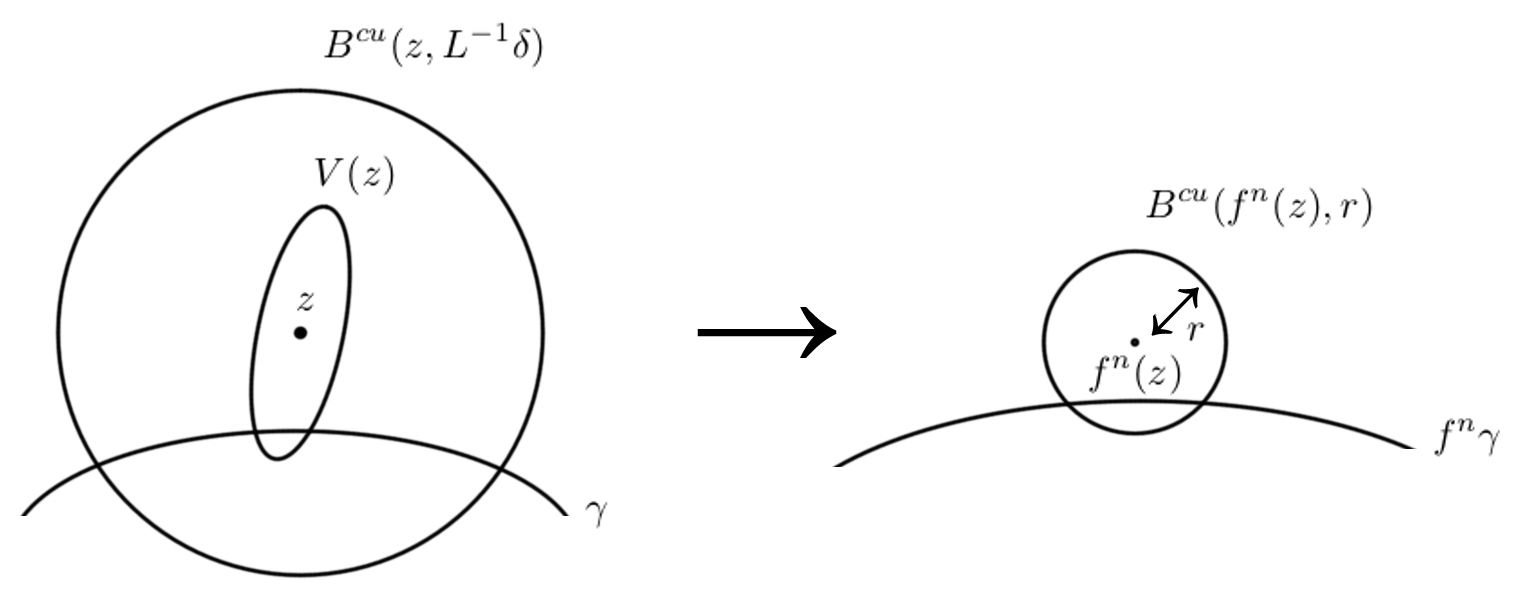}\\
\end{figure}

We want to bound $\frac{m_{\gamma}(V(z))}{m_{\gamma}(B^{cu}(z,L^{-1}\delta))}$, for this purpose we estimate $ \frac{m_{\R}(I)}{m_{\R}(J)}$.

Fix $R_{2}>0$ such that for every $r<R_{2}$ every unstable curve intersects $B^{cu}(\cdot , Lr)$ with length at most $3Lr$. Taking $r<R_{2}$ we see that:
\begin{align}
\nonumber 3Lr &\geq m_{f_{m}(\gamma)}(f_{m}\gamma\cap B^{cu}(f^{n}(z) , Lr) ) \\
\nonumber &\geq m_{f_{m}(\gamma)}(f_{m}\gamma \cap \pi_{j}^{-1} B^{cu}(\pi_{j} f^{n}(z),r)) \\
\nonumber &\geq m_{\gamma}(\gamma\cap f_{m}^{-1}\pi_{i}^{-1}B(x,r)) \cdot \min_{w\in\gamma}{||Df^{n}_{|_{E^{uu}_{w}}}||} \\
\nonumber &\geq m_{\R}(I)\cdot \min_{w\in C(z)}{||Df^{n}_{|_{E^{uu}_{w}}}||}
\end{align}

Then $ m_{\R}(I)\leq\frac{3Lr}{\underset{w\in C(z)}{\min}{||Df^{n}_{|_{E^{uu}_{w}}}||}}$.

When $m_{\gamma}(f_{m}^{-1}\pi_ {j}^{-1}B^{cu}(f^{n}(z),r))=0$, we have $m_{\R}(I)=0$ and the inequality of the lemma holds in this case. When $m_{\gamma}(f_{m}^{-1}\pi^{-1}B^{cu}(f^{n}(z),r))>0$, the choice of $\delta$ ($=10C^{-1}L^{2}(\lambda_{c}^{-})^{-n}r$) guarantees that
\begin{align*}
V(z)&\subset B^{cu}(z, \underset{w\in C(z) }{\min}{||Df^{-n}_{|_{E^{c}_{w}}}}||^{-1} Lr)\\
&\subset B^{cu}(z,C^{-1}(\lambda_{c}^{-})^{-n}Lr)\\
&\subset B^{cu}(z,\frac{1}{2}L^{-1}\delta)
\end{align*}

Take $r$ small enough such that $a_{1} \geq 10\delta$. Since $m_{\R}(\gamma)\geq a_{1} \geq 10\delta$ and $\gamma\cap B^{cu}(z,\frac{1}{10}L^{-1}\delta)\neq\emptyset$, we have that the length of $\gamma\cap B^{cu}(z,\delta)$ is greater than $ \frac{1}{2}L^{-1}\delta$, so $m_{\R}(J)\geq \frac{1}{2}L^{-1}\delta$. Then:
\begin{align*}
\frac{m_{\R}(I)}{m_{\R}(J)} \leq \frac{\frac{3Lr}{\min||Df^{n}_{|_{E^{uu}_{w}}}||}}{\frac{1}{2}L^{-1}\delta}
\leq \left(\frac{r}{\delta}\right)^{2} \frac{60 L^{4}C^{-1}}{(\lambda_{c}^{-})^{n}\cdot \underset{w\in C(z)}{\min}{|| Df^{n}_{|_{E^{uu}_{w}}}||}}
\end{align*}

Thus, the Lemma is valid in the case $\mu=m_{\gamma}$. 

Now, taking $\mu = \rho m_{\gamma}$, with $\gamma$ as before, $\log\rho$ of class $(C_{1},\alpha_{1})$-Holder, $\rho$ defined in an interval of length smaller than $a_{2}$,
 since  
$\frac{\int_{I}{\rho(x)dx}}{\int_{J}{\rho(x)dx}} \leq e^{2C_{1}a_{2}^{\alpha_{1}}} \frac{m_{\R}(I)}{m_{\R}(J)}$
for every $I\subset J\subset \R$, we have
$$\displaystyle \frac{(\rho m_{\gamma})(f_{m}^{-1}\pi^{-1}B(x,r))}{(\rho m_{\gamma})(\pi^{-1}B(\pi f_{m}^{-1}\pi^{-1}(x),\delta))}\leq\frac{e^{2C_{1}a_{2}^{\alpha_{1}}}\cdot{60 L^{4}C^{-1}}}{\sigma_{n,r}}\cdot \left(\frac{r}{\delta}\right)^{2}.
$$

Hence Lemma 4.2 also holds for measures $\mu=\rho m_{\gamma}$, with $B_{1}=e^{2C_{1}a_{2}^{\alpha_{1}}}\cdot{60 L^{4}C^{-1}}$.

In the case where $\mu$ is a finite sum $\mu=\sum_{i=1}^{s_{0}}{\rho_{i}m_{\gamma_{i}}}$, with $\rho_{i}$ and $\gamma_{i}$ as before, the result follows by linearity, since both sides of the inequality are linear in $\mu$.

Finally, if $\mu$ is an ergodic u-Gibbs, we proceed using Lemma 4.1. The measures $\mu_{n}$ given by Lemma 4.1 are finite sum of measures of the type $\rho m_{\gamma}$, so the desired inequality holds for $\mu_{n}$. If we take $r$ smaller than the inverse of the curvature of every unstable curve, then $\partial B^{cu}(\cdot,r)$ intersects every unstable curve at most finitely; so $\mu(\partial B^{cu}(\cdot,r))=0$ for such $r$. Since $\mu$ is the limit of $\mu_{n}$, we can take the limit in each term of the inequality, so what we desired also holds for $\mu$.
\end{proof}

\subsection{One Inequality for the Norm}

\begin{lemma}
There are two constants $B_{2}>0$ and $\sigma_{2}>1$ such that for every $n$, $C_i$ and $C_j$ there exists a constant $R_{n}>0$ such that, for $\delta=10CL(\lambda_{c}^{-})^{-n}r$, for every $r<R_{n}$ and any ergodic u-Gibbs $\mu $:
\begin{align*}
\sum_{m}{||(\pi_{j})_{*}(f_{m})_{*}\mu||_{W_{j},r}^{2}} \leq \frac{B_{2}}{\sigma_{2}^{n}} ||(\pi_{i})_{*}\mu||_{\tilde{W}_{i},\delta}^{2}
\end{align*}
\end{lemma}

\begin{proof}
Applying Lemma 4.2 to the measure $\mu$, we see that:
\begin{align}
\nonumber ||(\pi_{j})_{*}(f_{m})_{*}\mu ||^{2}_{W_{i},r} &= \frac{1}{r^{4}}\int_{W_{j}}{\mu(f_{m}^{-1}\pi_{j}^{-1}B^{cu}(y,r))^{2} \d m^{cu}(y)}\\
\nonumber &\leq \frac{B_{1}^{2}}{\delta^{4}} \underset{W_{j}}{\int}{\frac{\mu(\pi_{i}^{-1}B(\pi_{i} f_{m}^{-1}\pi_{j}^{-1}(y),\delta))^{2}}{(\lambda_{c}^{-})^{2n}\cdot\underset{w\in C(y)}{\min}{||Df^{n}_{|_{E^{uu}_{w}}}||^{2}}} \d m^{cu}(y)}
\end{align}

Changing the variables $\tilde{y}=\pi_{i}\circ f_{m}^{-1}\circ\pi_{j}^{-1}(y)$ and noticing that $\pi_{i}\circ f_{m}^{-1}\circ\pi_{j}^{-1}: W_{j}\rightarrow \tilde{W}_{i}$ is an absolutely continuous homeomorphism into the image with Jacobian bounded by $\displaystyle L^{2}\cdot \sup_{x\in W^{ss}(y)}{|\det Df^{n}_{|_{E^{cu}_{x}}}|}$, we have:

\begin{align*}
      ||(\pi_{j})_{*}(f_{m})_{*}\mu ||^{2}_{W_{i},r} &\leq \frac{B_{1}^{2}}{\delta^{4}} \underset{\tilde{W}_{i}\cap R_{i,j,m}}{\int}{\frac{\mu(\pi_{i}^{-1}B(\tilde{y},\delta))^{2} \cdot \Jac(\pi_{i} f_{m}\pi_{j}^{-1}(y)) }{(\lambda_{c}^{-})^{2n}\underset{w\in\pi^{-1}B(\tilde{y},L\delta)}{\min}{||Df^{n}_{|_{E^{uu}_{w}}}||^{2}}} \d m^{cu}(\tilde{y})}  \displaybreak[1]     \\ 
       &\leq \frac{B_{1}^{2}L^{2}}{\delta^{4}} \underset{\tilde{W}_{i}\cap R_{i,j,m}}{\int}{    \frac{\mu(\pi_{i}^{-1}B(\tilde{y},\delta))^{2} \cdot \underset{y_{1}\in W^{ss}_{\tilde{x}}}{\sup}{|\det Df^{n}_{|_{E^{cu}_{y_{1}}}}|}}{(\lambda_{c}^{-})^{2n}\underset{w\in\pi^{-1}B(\tilde{x},L\delta)}{\min}{||Df^{n}_{|_{E^{uu}_{w}}}||^{2}}} \d m^{cu}(\tilde{y})}
\end{align*}

To estimate the term inside the integral we use the central direction neutral condition. For that, we use the following Claim, which will be proved later.

\begin{claim} There exists a constant $K_{2}>0$ such that for every $n\in\N$ there exists $\tilde{r}_{n}>0$ such that for every $r<\tilde{r}_{n}$ it holds:
\begin{align*}
\frac{\underset{y_{1}\in W^{ss}_{\tilde{x}}}{\sup}{|\det Df^{n}_{|_{E^{cu}_{y_{1}}}}|}}{(\lambda_{c}^{-})^{2n}\underset{w\in\pi^{-1}B^{cu}(\tilde{x},L\delta)}{\min}{||Df^{n}_{|_{E^{uu}_{w}}}||^{2}}}\leq K_{2}\left[\frac{(\lambda_{c}^{+})}{(\lambda_{c}^{-})^{2}(\lambda_{u}^{-})}\right]^{n}
\end{align*}
\end{claim}

Supposing that $r<\tilde{r}_{n}$ and considering $\sigma_{2}^{-1}=\frac{\lambda_{c}^{+}}{(\lambda_{c}^{-})^{2}(\lambda_{u}^{-})}<1$, then:
\begin{align}
\nonumber ||(\pi_{j})_{*}(f_{m})_{*}\mu ||^{2}_{W_{i},r} &\leq \frac{B_{1}^{2}L^{2}K_{2}}{\sigma_{2}^{n}} \underset{\tilde{W}_{i}\cap R_{i,j,m}}{\int}{\frac{\mu(\pi_{i}^{-1}B^{cu}(\tilde{x},\delta))^{2}}{\delta^{4}} \d m^{cu}(\tilde{x})}
\end{align}

Defining $B_{2}=B_{1}^{2}L^{2}K_{2}$ and adding this inequality in $m$, we get:
\begin{align*}
\sum_{m}{||(\pi_{j})_{*}(f_{m})_{*}\mu||^{2}_{W_{i},r}} &\leq \frac{B_{2}}{\sigma_{2}^{n}}\sum_{m}{\underset{\tilde{W}_{i}\cap R_{i,j,m}}{\int}{\frac{\mu(\pi_{i}^{-1}B^{cu}(\tilde{x},\delta))^{2}}{\delta^{4}} \d m^{cu}(\tilde{x})}} \\
&= \frac{B_{2}}{\sigma_{2}^{n}}\underset{\tilde{W}_{i}}{\int}{\frac{(\pi_{i})_{*}\mu(B(\tilde{x},\delta))^{2}}{\delta^{4}} \d m^{cu}(\tilde{x})} \\
&=\frac{B_{2}}{\sigma_{2}^{n}} ||(\pi_{i})_{*}\mu||^{2}_{\tilde{W}_{i},\delta}
\end{align*}
\end{proof}

Let us now prove Claim 4.1.

\begin{proof}[Proof of Claim 4.1]
We first consider the case where $y_{1}\in W^{ss}_{\tilde{x}}$ and $w=\tilde{x}$.

There exists $\tilde{K}_{2}>0$ such that $ \frac{|det Df^{n}_{|_{E^{cu}_{y_{1}}}}|}{|det Df^{n}_{|_{E^{cu}_{y_{2}}}}|} \leq \tilde{K}_{2}$ for every $y_{2}\in W^{ss}_{y_{1}}$ and every $n\geq 0$, this is due to the following:
\begin{align}
\nonumber \log|\det Df^{n}_{|_{E^{cu}_{y_{1}}}}| - \log|\det Df^{n}_{|_{E^{cu}_{y_{2}}}}| &= \sum_{j=1}^{n}{\log|\det Df_{|_{E^{cu}_{f^{j}(y_{1})}}}| - \log|\det Df_{|_{E^{cu}_{f^{j}(y_{2})}}}|} \\
\nonumber &\leq \sum_{j\geq 0}{\tilde{C}_{2} d(f^{j}(y_{1}),f^{j}(y_{2}))^{\alpha_{2}}} \\
\nonumber &\leq \sum_{j=0}^{+\infty}{\tilde{C}_{2} l ((\lambda_{ss})^{\alpha_{2}})}^{j} := \tilde{K}_{2}
\end{align}

Here we have used that the function $x \rightarrow \log|\det Df_{|_{E^{cu}_{x}}}|$ is $(\tilde{C}_{2},\alpha_{2})$-Holder and that the length of $W^{ss}$ is uniformly bounded by $l$ inside the boxes. We then have:

\begin{align*}
\frac{|\det Df^{n}_{|_{E^{cu}_{y_{1}}}}|}{(\lambda_{c}^{-})^{2n}\cdot ||Df^{n}_{|_{E^{uu}_{\tilde{x}}}}||^{2}} &\leq \tilde{K}_{2}\cdot \left[ \frac{|\det Df^{n}_{|_{E^{cu}_{\tilde{x}}}}|}{(\lambda_{c}^{-})^{2n}\cdot ||Df^{n}_{|_{E^{uu}_{\tilde{x}}}}||^{2}} \right]\\
&=  \frac{\tilde{K}_{2} \cdot ||Df^{n}_{|_{E^{c}_{\tilde{x}}}}|| \cdot ||Df^{n}_{|_{E^{uu}_{\tilde{x}}}}||}{(\lambda_{c}^{-})^{2n} ||Df^{n}_{|_{E^{uu}_{\tilde{x}}}}||^{2}} \\
&\leq \frac{C^{2}\tilde{K}_{2}(\lambda_{c}^{+})^{n}}{C^{-1}(\lambda_{c}^{-})^{2n} (\lambda_{u}^{-})^{n}} \\
&\leq \frac{C^{3}\tilde{K}_{2}}{\sigma_{2}^{n}}
\end{align*}

Now, if $y_{1}\in W^{ss}_{\tilde{x}}$ and $w\in \pi_{j}^{-1} B^{cu}(\tilde{x},L\delta)$, by continuity, there exists $\tilde{r}_{n}=\tilde{r}_{n}(w_{2})$ such that $\frac{1}{2} \leq \frac{||Df^{n}_{|_{E^{uu}_{w_{1}}}}||}{||Df^{n}_{|_{E^{uu}_{w_{2}}}}||}\leq 2$ for every $w_{1}\in B^{cu}(w_{2},r_{n}(w_{2}))$. By compactness, $\tilde{r}_{n}$ can always be taken to be uniform in $w_{2}$. Taking $\tilde{r}_{n}$ small enough, for every $r<\tilde{r}_{n}$, every $w\in B^{cu}(z,\tilde{r}_{n})$ and any $z\in W^{ss}_{\tilde{x}}=W^{ss}_{y_{1}}$, we have:
\begin{align*}
\frac{|\det Df^{n}_{|_{E^{cu}_{y_{1}}}}|}{(\lambda_{c}^{-})^{2n}\cdot ||Df^{n}_{|_{E^{uu}_{w}}}||^{2}} \leq \frac{C^{3}\tilde{K}_{2}\cdot |\det Df^{n}_{|_{E^{cu}_{z}}}|}{(\lambda_{c}^{-})^{2n}\cdot \frac{1}{2} ||Df^{n}_{|_{E^{uu}_{z}}}||^{2}} \leq \frac{K_{2}}{\sigma_{2}^{n}}
\end{align*}
\end{proof}

\subsection{Inequalities using the Transversality}

\begin{lemma}
For every $\theta\in(0,\frac{\pi}{2})$ there exist constants $r_{\theta}$ and $C_{\theta}$ for which the following is valid: given Lipschitz curves $\gamma_{1}$ and $\gamma_{2}$ of finite length contained in some center-unstable manifold contained in $V$ that are $\theta$-transversal in neighborhoods of radius $r_{\theta}$, for every $r<\frac{r_{\theta}}{4}$ and every connected component $E_{r}$ of $\big\{ s, |d(\gamma_{1}(s),\gamma_{2})|\leq r \big\}$, the length of $\gamma_{1}(E_{r})$ is at most $C_{\theta} r$.
\end{lemma}

\begin{proof}
We consider first the case where $\gamma_1$ and $\gamma_2$ are curves contained in $\R^{2}$.

If the curves $\gamma_{1}$ and $\gamma_{2}$ are contained in $\R^{2}$, for every point $\gamma_{2}(t)$ we consider cones $C_{1}(t)$ and $C_{2}(t)$ as in the definition of $\theta$-transversality. By transversality in neighborhoods of radius $r$, the curves intersects at most once on neighborhoods of radius smaller than $\frac{r}{2}$. We also suppose that the extremal line of $C_{2}(t)$ closest to $\gamma_1$ is contained in the axes $x$ and that ${\gamma_{1}}_{|_{\gamma_{1}^{-1}E_{r}}}$ is a graph over it. We will consider ${\gamma_{1}}({|_{\gamma_{1}^{-1}E_{r}}})$ as a graph over $\gamma_{2}^{-}$.
\begin{figure}[h]
  \psfrag{A}{$\theta$}
  \includegraphics[width=0.6\textwidth]{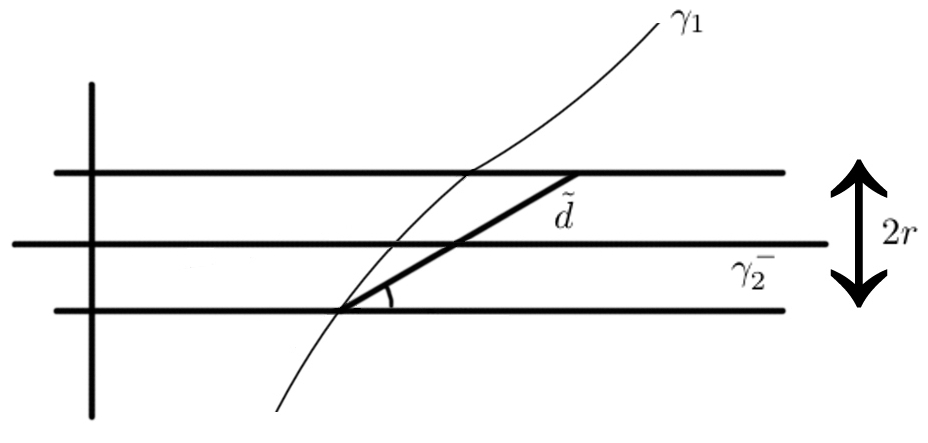}
  \psfrag{b}{by}
  \psfrag{ca}{oi}
\end{figure}

Notice that $\gamma_{1}(t)$ increases at least $C_{\theta} t$. Actually, the length of $\tilde{d}$ is at most $\frac{2r}{\sin(\theta)}$, and this geometric comparison passes to the following inequality:
\begin{align*}
 l(\gamma_{1}(E_{r})) &\leq m_{\R}\big(\{ s, d(\gamma_{1}^{+}(s),\gamma_{2}^{-})\leq r \}\big)\\
 &\leq \frac{2}{\tan(\theta)} r = C_{\theta} r
\end{align*}

Now, suppose that the curves are contained in the same center-unstable manifold. Let $U_{i}$ be the sets given by Proposition 2.4 and the open set $\displaystyle V=\cup_{i}{U_{i}}$ that contains $\Lambda$,
since the family of diffeomorphisms $\{ \psi_{i} \}_{i\in I}$ given by Proposition 2.4 have derivative uniformly bounded, taking $T=\underset{i\in I}{\sup}||D{\psi_{i}}_{|_{_{W^{cu}_{z}}}}||$, the length of $\gamma_{1}(E_{r})$ is smaller than the length of $\psi_{i}(\gamma_{1}(E_{r}))$ times $T$. Let $r_{\theta}$ be such that $C_{\theta}r_{\theta} T < \frac{R_{1}}{2}$, then $\gamma_{1}(E_{r})$ is contained in some $U_{i_{0}}$, $i_{0}\in I$. The result follows taking $r<\min\{R_{1}, r_{\theta} \}$ and noticing that in $V$ the distances $d^{cu}(x,y)$ and $d(x,y)$ are equivalent, and that $\{ s, d(\gamma_{1}(s),\gamma_{2})\leq r\} \subset \{ s , d(\psi_{i_{0}}\gamma_{1}(s), \psi_{i_{0}}\gamma_{2}) \leq Tr \}$.
\end{proof}

The transversality between stable projections of unstable curves can be used to estimate the bilinear form of stable projections of u-Gibbs.

\begin{lemma}
Given $n$, $C_{i}$ and $C_{j}$, there exist $\hat{r}_{n}>0$ and $A_{n}>0$ such that for every ergodic u-Gibbs $\mu$, $r<\hat{r}_{n}$ and for every $m \neq m'$, the following holds:
\begin{equation*}
\langle (\pi_{j})_{*}(f_{m})_{*}\mu , (\pi_{j})_{*}(f_{m'})_{*}\mu \rangle_{W_{j}, r} \leq A_{n}\cdot | (\pi_{j})_{*}(f_{m})_{*}\mu | \cdot | (\pi_{j})_{*}(f_{m'})_{*}\mu |
\end{equation*}
\end{lemma}

\begin{proof}
By finiteness, there exist constants $d_{1}>0$ and $d_{2}>0$ such that for every pair of components $f^{n}(C_{i,j,m})$ and $f^{n}(C_{i,j,m'})$, $m\neq m'$, it is valid that either $d^{ss}(R_{i,j,m},R_{i,j,m'})>d_{1}$ or $d^{cu}(R_{i,j,m},R_{i,j,m'})>d_{2}$. Consider $\theta_{n}=\theta(d_{1})$, $\hat{R}_{n}=r(d_{1})$ given by the transversality condition and $r<\min\{\frac{\hat{R}_{n}}{4},\frac{d_{2}}{4}\}$. 

Lemma 4.5 will be proved first for measures supported on single unstable curves and after for approximations of $\mu$ by combination of measures supported in unstable curves.

Suppose $\mu_{1}=\rho_{1} m_{\gamma_{1}}$ and $\mu_{2}=\rho_{2} m_{\gamma_{2}}$, where $\log\rho_{i}$ is $(C_{1},\alpha_{1})$-Holder and $\gamma_{i}$ is an unstable curve contained in $C_{i}$ that crosses $C_{i}$, $i=1,2$. Consider the Lipschitz curves $\tilde{\gamma}_{m} = \pi_{j}\circ f_{m} (\gamma_{1})$ and $\tilde{\gamma}_{m'} = \pi_{j}\circ f_{m'} (\gamma_{2})$, the measures $\nu_{m}=(\pi_{j})_{*}(f_{m})_{*}\mu_{1}$ and $\nu_{m'}=(\pi_{j})_{*}(f_{m'})_{*}\mu_{2}$ supported on these curves, consider the densities $\rho_{m}=\frac{d\nu_{m}}{d m_{\tilde{\gamma}_{m}}}=(\Jac \pi_{j}^{ss})^{-1} (\Jac f_{m})^{-1} \rho_{1}$ and $\rho_{m'}=\frac{d\nu_{m'}}{d m_{\tilde{\gamma}_{m'}}}=(\Jac \pi_{j}^{ss})^{-1} (\Jac f_{m'})^{-1} \rho_{2}$. In this case we will prove that $\langle \nu_{m} , \nu_{m'} \rangle_{W_{j},r}\leq C_{\theta_{n}} |\nu_{m}| |\nu_{m'}|$.

There exist constants $C_{3}$ and $\alpha_{3}$ (depending on $C_{1}$, $\alpha_{1}$, $L$) such that the densities $\rho_{m}$ and $\rho_{m'}$ are $(C_{3},\alpha_{3})$-Holder functions. Since the curves $\tilde{\gamma}_{*}$ have length uniformly bounded by $L^{-1}a_{1}$ and $L a_{2}$, we have ${\rho_{*}(\gamma_{*}(x))}\leq e^{C_{3}(L a_{2})^{\alpha_{3}}}{\rho_{*}(\gamma_{*}(y))}$, $*\in\{m,m'\}$. Let  $C_{4}$  be a constant such that $m^{cu}(B^{cu}(z,r))\leq C_{4}r^{2}$ for every $z \in\cup{C_{i}}$.

Define the function $\mathbbm{1}_{r}$ by $\mathbbm{1}_{r}(z_{1},z_{2})=1$ if $d^{cu}(z_{1},z_{2})<r$ and $\mathbbm{1}_{r}(z_{1},z_{2})=0$ otherwise.

If $d^{cu}(R_{i,j,m},R_{i,j,m'})<r $ then $\langle \nu_{m} , \nu_{m'} \rangle_{W_{j},r}$ is zero and the inequality is valid. Else, $d^{ss}(R_{i,j,m},R_{i,j,m'})>d_{1}$, and so:

\begin{align}
\nonumber \langle \nu_{m}, \nu_{m'} \rangle_{W_{j},r} &= \frac{1}{r^{4}}\underset{W_{j}}{\int}{\nu_{m}(B(z,r)) \nu_{m'}(B(z,r)) \d m^{cu}(z)} \\
\nonumber &=\frac{1}{r^{4}} \underset{W_{j}}{\int} \Big(\underset{\tilde{\gamma}_{m}\cap B(z,r)}{\int}{\rho_{m}(\tilde{\gamma}_{m}(s)) \d m_{\tilde{\gamma}_{m}}(s)}\Big)\\
\nonumber &\quad\quad\quad\quad\quad\quad  \Big(\underset{\tilde{\gamma}_{m'}\cap B(z,r)}{\int}{\rho_{m'}(\tilde{\gamma}_{m'}(t)) \d m_{\tilde{\gamma}_{m'}}(t)}\Big) \d m^{cu}(z) \\
\nonumber &= \frac{1}{r^{4}}  \underset{W_{j}}{\int} \underset{\tilde{\gamma}_{m}}{\int} \underset{\tilde{\gamma}_{m'}}{\int} \Big( \rho_{m}(\tilde{\gamma}_{m}(s)) \rho_{m'}(\tilde{\gamma}_{m'}(t))  \mathbbm{1}_{r}(\tilde{\gamma}_{m}(s),z) \mathbbm{1}_{r}(\tilde{\gamma}_{m'}(t),z) \Big) \d s \d t \d z\\
%
\nonumber &\leq \frac{1}{r^{4}} \underset{\tilde{\gamma}_{m'}}{\int} \rho_{m'}(t) \bigg( \underset{\tilde{\gamma}_{m'}}{\int}\rho_{m}(s)\cdot \mathbbm{1}_{2r}(\tilde{\gamma}_{m}(s), \tilde{\gamma}_{m'}(t) )\\
\nonumber &\quad\quad\quad\quad\quad\quad \Big(\underset{W_{j}}{\int}\mathbbm{1}_{r}(\tilde{\gamma}_{m}(s),z) \d z\Big) \d s \bigg) \d t\\
\nonumber \quad\quad\quad &\leq \frac{C_{4}}{r^{2}} \underset{\tilde{\gamma}_{m}}{\int} \underset{\tilde{\gamma}_{m'}}{\int} \Big(\rho_{m}(\tilde{\gamma}_{m}(s)) \rho_{m'}(\tilde{\gamma}_{m'}(t))  \mathbbm{1}_{2r}( \tilde{\gamma}_{m}(s) , \tilde{\gamma}_{m'}(t) )\Big) \d s \d t\\
\nonumber \quad\quad\quad &\leq \frac{C_{4} e^{2C_{3}(La_{2})^{\alpha_{3}}}}{r^{2}} \rho_{m}(\gamma_{m}(s_{0})) \rho_{m'}(\gamma_{m'}(t_{0}))\\
\nonumber \quad\quad\quad &\quad\quad\quad\quad\quad\quad \Big(\underset{\tilde{\gamma}_{m}}{\int} \mathbbm{1}_{2r}(\tilde{\gamma}_{m}(s), \tilde{\gamma}_{m'}) \d s \Big) \Big(\underset{\tilde{\gamma}_{m'}}{\int} \mathbbm{1}_{2r}(\tilde{\gamma}_{m}, \tilde{\gamma}_{m'}(t)) \d t \Big)
\end{align}

To estimate the last term in parentheses, we will use Lemma 4.4. Take $r<R_{\theta_n}$ small such that Lemma 4.4 holds and consider the constant $C_{\theta_{n}}$ given by that lemma. Note also that for these projections of unstable curves, by transversality in neighborhoods with radius smaller than $\hat{R}_{n}$, there exists an integer $M_{n}\in\N$ such that each $\tilde{\gamma}_{m}$ intersects the neighborhood of radius $r$ of $\tilde{\gamma}_{m'}$ in at most $M_{n}$ connected components. Then:
\begin{align}
\nonumber \langle \nu_{m}, \nu_{m'} \rangle_{W_{j},r} &\leq \frac{C_{5}}{r^{2}}  \rho_{m}(\gamma_{m}(s_{0})) \rho_{m'}(\gamma_{m'}(t_{0})) \cdot (M_{n}C_{\theta_{n}} r)^{2} \\
\nonumber &= (C_{5} (M_{n}C_{\theta_{n}})^{2}) \cdot \rho_{m}(\gamma_{m}(s_{0})) \rho_{m'}(\gamma_{m'}(t_{0}))
\end{align}

In what follows, we will use that $\underset{\tilde{\gamma}_{m}}{\int}{\rho_{m}(\gamma_{m}(s)) \d s} \geq K_{3} \rho_{m}(\gamma_{m}(s_{0}))$ for some constant $K_{3}>0$. This is due to a simple calculation:
\begin{align}
\nonumber \underset{\tilde{\gamma}_{m}}{\int}\rho_{m}(\gamma_{m}(s)) \d s &\geq \underset{\tilde{\gamma}_{m}}{\int} e^{-(C_{2}(La_{2})^{\alpha_{2}})} \rho_{m}(\gamma_{m}(s_{0})) \d s\\
\nonumber &= e^{-(C_{2}(La_{2})^{\alpha_{2}})}  \rho_{m}(\gamma_{m}(s_{0}))  l(\tilde{\gamma}_{m})\\
\nonumber &\geq \left(L^{-1} a_{1} e^{-(C_{2}(La_{2})^{\alpha_{2}})} \right) \rho_{m}(\gamma_{m}(s_{0})) \\
\nonumber &= K_{3}\rho_{m}(\gamma_{m}(s_{0})).
\end{align}

Analogously, it is valid that $\rho_{m'}(\gamma_{m'}(t_{0})) \leq K_{3}^{-1} \underset{\tilde{\gamma}_{m'}}{\int}\rho_{m'}(\gamma_{m'}(t)) \d t$, so we have the estimate:
\begin{align}
\nonumber \langle \nu_{m}, \nu_{m'} \rangle_{W_{j}, r} &\leq \left(C_{5} (M_{n}C_{\theta_{n}})^{2} (K_{3})^{2}\right) \underset{\tilde{\gamma}_{m}}{\int}{\rho_{m}(\gamma_{m}(s)) \d s} \underset{\tilde{\gamma}_{m'}}{\int}{\rho_{m'}(\gamma_{m'}(t)) \d t}\\
\nonumber &= A_{n} \cdot |\nu_{m}| \cdot |\nu_{m'}|
\end{align}

The second case to be considered is when the measure $\mu$ is a finite sum $ \mu=\sum_{k=1}^{s_{0}}{\rho_{k} m_{\gamma_{k}}}$. In this case, the inequality holds by linearity because we also have $\theta_{n}$-transversality of the stable projection of unstable curves. Actually, taking $\nu_{m}=(\pi_{j})_{*}(f_{m})_{*}\mu$ and $\nu_{m'}=(\pi_{j})_{*}(f_{m'})_{*}\mu$ we see that:
\begin{align*}
\langle \nu_{m} , \nu_{m'} \rangle_{W_{j},r} &= \langle (\pi_{j})_{*}(f_{m})_{*}(\sum_{k}{\rho_{k}m_{\gamma_{k}}}) , (\pi_{j})_{*}(f_{m'})_{*}(\sum_{k'}{\rho_{k'}m_{\gamma_{k'}}}) \rangle_{W_{j},r} \\
&= \sum_{k,k'} \langle (\pi_{j})_{*}(f_{m})_{*}({\rho_{k}m_{\gamma_{k}}}) , (\pi_{j})_{*}(f_{m'})_{*}({\rho_{k'}m_{\gamma_{k'}}}) \rangle_{W_{j},r} \\
&\leq \sum_{k,k'} A_{n} | (\pi_{j})_{*}(f_{m})_{*}({\rho_{k}m_{\gamma_{k}}}) | | (\pi_{j})_{*}(f_{m'})_{*}({\rho_{k'}m_{\gamma_{k'}}}) |\\
&= A_{n} | (\pi_{j})_{*}(f_{m})_{*}\mu | | (\pi_{j})_{*}(f_{m'})_{*}\mu |\\
&= A_{n} | \nu_{m} | | \nu_{m'} |
\end{align*}

Finally, let us suppose that $\mu$ is an ergodic u-Gibbs. In this case, we use Lemma 4.1 to approximate $\mu$ by probability measures $\sum_{i}\rho_{i}m_{\gamma_{i}}$ and apply Lemma 3.3 to approximate to the bilinear form on the left-hand side of the inequality.
\end{proof}

\subsection{Proof of the Main Inequality}

Let us give a Localized Version of the Main Inequality when two boxes $C_{i}$, $C_{j}$ and an integer $n$ are fixed.

\begin{prop}[Main Inequality - Localized Version]
There exist $\tilde{B}>0$ and $\sigma>1$, such that for every $n\in\mathbb{N}$ and $C_{i}$, $C_{j}$ fixed, there exist $\tilde{D}_{n}>0$, $r_{n}>0$ and $c_{n}>1$ such that for every ergodic u-Gibbs $\mu$ and any $r<r_{n}$, it holds:
\begin{align}
\nonumber \left|\left| (\pi_{j})_{*}(f^{n}_{*}(\mu_{|_{C_{i}}})) \right|\right|^{2}_{W_{j},r} \leq \frac{\tilde{B}}{\sigma^{n}} \left|\left|(\pi_{i})_{*}\mu \right|\right|^{2}_{\tilde{W}_{i},c_{n}r} + \tilde{D}_{n} \left|(\pi_{i})_{*}\mu\right|^{2} 
\end{align}
\end{prop}

\begin{proof}[Proof of the Localized Version]

Take $r_{n}$ small such that Lemmas 4.3 and 4.5 holds for every $r<r_{n}$. Then:
\begin{align}
\nonumber ||(\pi_{j})_{*}(f^{n}_{*}\small(\mu_{|_{C_{i}}}\small))|&|_{W_{j},r}^{2} = ||(\pi_{j})_{*}\sum_{m}(f_{m})_{*}\mu||^{2}_{W_{j},r}\\
\nonumber &= \sum_{m}||(\pi_{j})_{*}(f_{m})_{*}\mu||^{2}_{W_{j},r} + \sum_{m\neq m'}\langle (\pi_{j})_{*}(f_{m})_{*}\mu , (\pi_{j})_{*}(f_{m'})_{*}\mu \rangle_{W_{j},r}\\
\nonumber &\leq \frac{B_{2}}{\sigma^{n}}||(\pi_{i})_{*}\mu||_{\tilde{W}_{i},c_{n}r}^{2} + \sum_{m\neq m'} A_{n} |(\pi_{j})_{*}(f_{m})_{*}\mu| |(\pi_{j})_{*}(f_{m'})_{*}\mu| \\
\nonumber &= \frac{\tilde{B}}{\sigma^{n}}||(\pi_{i})_{*}\mu||_{\tilde{W}_{i},c_{n}r}^{2} + \tilde{D}_{n} |(\pi_{i})_{*}\mu|^{2}
\end{align}
\end{proof}


\begin{proof}[Proof of the Main Inequality]
Note that $|||(f_{n;i,j,m})_{*}\mu|||_{\{ \mathcal{W} \}, r}=||(\pi_{j})_{*}\mu||_{W_{j},r}$  , where $\mathcal{W}=\{ W_1,\cdots, W_{s_0} \}$ is the finite family of boxes fixed in the beginning of this Section, and note also that $f^{n}_{*}\mu (E) \leq \sum_{i,j,m}{(f_{m})_{*}\mu (E)}$ for every measurable set $E$. Then:
\begin{align}
\nonumber |||f^{n}_{*}\mu |||^{2}_{\{ \mathcal{W} \},r} &\leq \sum_{i,j}\left|\left| (\pi_{j})_{*}(f^{n}_{*}\small(\mu_{|_{C_{i}}}\small)) \right|\right|^{2}_{W_{j},r} \\
\nonumber &\leq \sum_{i,j} \left[\frac{\tilde{B}}{\sigma^{n}}||(\pi_{i})_{*}\mu||^{2}_{\tilde{W}_{i},c_{n}r} + \tilde{D}_{n} |(\pi_{i})_{*}\mu|^{2}\right]\\
\nonumber &\leq \sum_{i,j} \left[\frac{\tilde{B}}{\sigma^{n}}|||\mu|||^{2}_{\{ \tilde{\mathcal{W}}\}, c_{n}r} + \tilde{D}_{n} |\mu|^{2}\right]\\
\nonumber &\leq \frac{\tilde{B} s_{0}^{2} }{\sigma^{n}}|||\mu|||_{\{ \mathcal{W} \},c_{n}r}^{2} + \tilde{D}_{n}s_{0}|\mu|^{2}\\
\nonumber &\leq \frac{B}{\sigma^{n}}|||\mu|||^{2}_{\{\mathcal{W} \},c_{n}r} + D_{n}|\mu|^{2}
\end{align}
\end{proof}

\section{Physical Measures}\label{s.physmeas}

Through this Section we will assume that $f$ satisfies the assumptions of Theorem A and the same boxes and norms considered in the Lemmas along Section 4.

\subsection{Existence of Physical Measures}

We will prove that every u-Gibbs projects by the stable projection into absolutely continuous measures in the center-unstable manifolds $W_{i}$  and that this fact implies positive measure for the basin of these measures.

\begin{prop}
Every ergodic u-Gibbs projects into an absolutely continuous measure in $\tilde{W}_i$ by the applications $\pi_{i}$. Moreover, for every ergodic u-Gibbs $\mu$ there exists a constant $L>0$ such that $ \left|\left|\frac{d((\pi_{i})_{*}\mu)}{dm^{cu}}\right|\right|_{L^{2}}\leq L$.
\end{prop}
\begin{proof}
Given the ergodic u-Gibbs $\mu$, consider $B$ and $\sigma$ as in the Main Inequality. Fix $N$ such that $\frac{B}{\sigma^{N}}<1$ and consider $D_{N}$, $r_{N}$ and $c_{N}$ as in the Main Inequality.
Then, for $r<r_{N}$:
\begin{align*}
|||\mu|||_{r}^{2}=|||f^{n}_{*}\mu|||^{2}_{r}\leq \frac{B}{\sigma^{N}}|||\mu|||^{2}_{c_{n}r}+ D_{N}|\mu|^{2}
\end{align*}

For $r<\frac{r_{N}}{c_{N}}$, we have:
\begin{align*}
 |||\mu|||_{c_{N}^{-1}r}^{2}\leq \frac{B}{\sigma^{n}}|||\mu|||^{2}_{r}+ D_{N}|\mu|^{2}
\end{align*}

Define the constant $K_{0}:= |||\mu|||_{\frac{r_{N}}{2}}^{2}$ and iterate $j$ times the inequality:

\begin{align}
\nonumber |||\mu|||_{c_{N}^{-j}\frac{r_{N}}{2}}^{2} &\leq \left(\frac{B}{\sigma^{N}}\right) |||\mu|||^{2}_{c_{N}^{-(j-1)}\frac{r_{N}}{2}}+ D_{N}|\mu|^{2} \\
\nonumber &\leq \cdots \\
\nonumber &\leq \left(\frac{B}{\sigma^{N}}\right)^{j}|||\mu|||^{2}_{\frac{r_{N}}{2}}+ D_{N}\left( 1+\frac{B}{\sigma^{N}}+\cdots+\left(\frac{B}{\sigma^{N}}\right)^{j-1} \right)|\mu|^{2}\\
\nonumber &\leq 1\cdot K_{0} + D_{N}\cdot\frac{1}{1-\left(\frac{B}{\sigma^{N}}\right)} =: \tilde{L}
\end{align}

Hence, for every $\pi_i$ we have:
\begin{align*}
\underset{r\rightarrow 0^{+}}{\liminf}{||(\pi_{i})_{*}\mu||_{W_{i},r}} \leq \underset{j\rightarrow\infty}{\liminf}{|||\mu|||^{2}_{c_{N}^{-j}\frac{r_{N}}{2}}}\leq \tilde{L} <+\infty ,
\end{align*}
and the criteria of absolute continuity for measures (Lemma 3.3) gives that
\begin{equation*}
(\pi_i)_{*}\mu \ll m^{cu} \quad\text{and}\quad \left|\left|\frac{d((\pi_i)_{*}\mu)}{dm^{cu}}\right|\right|_{L^{2}}\leq L.
\end{equation*}
\end{proof}


Since we have fixed the boxes $\{(C_{i},W_{i},\tilde{W}_{i},\pi_{i})\}$ that covers $\Lambda$, for every u-Gibbs $\mu$ there exists some $i_0$ such that $\mu(C_{i_{0}})>\frac{1}{2s_{0}}$.

\begin{prop}
Let  $\mu$ be an ergodic u-Gibbs  for $f$ and let $i$ be such that $\mu(C_{i})>\frac{1}{2s_{0}}$. If the measure $\nu=(\pi^{ss}_{i})_{*}\mu$ is absolutely continuous with respect to the Lebesgue measure $m^{cu}$ and $\left|\left|\frac{d\nu}{d m^{cu}}\right|\right|_{L^{2}}\leq K$, then $\mu$ is a physical measure and there exists a constant $M>0$ depending on $K$ such that $m(B(\mu))\geq M$.
\end{prop}

\begin{proof}
Since $\mu$ is ergodic, we have that $\mu(B(\mu))=1$. Note also that $\pi_{i}(B(\mu) )$ is measurable (Theorem 3.23 in \cite{measurability} guarantees the measurability of this projection). Then,
$$\nu( \pi_{i}B(\mu) ) = \mu((\pi_{i})^{-1}\pi_{i}B(\mu) ) \geq \mu(B(\mu)\cap C_{i}) >\frac{1}{2 s_{0}}$$
By the absolute continuity of $\nu$, it follows that $m^{cu}( \pi_{i}B(\mu) ) > 0$. Actually, we have $m^{cu}( \pi_{i}B(\mu) ) > (2s_{0}K)^{-1}$. In fact,
\begin{align}
\nonumber (2s_{0})^{-1} &< \nu( \pi_{i}B(\mu) )\\
\nonumber &= \int{ \mathbbm{1}_{\pi_{i}B(\mu)}}  \frac{d\nu}{dm^{cu}} dm^{cu}\\
\nonumber &\leq || \mathbbm{1}_{\pi_{i}B(\mu)} ||_{L^{2}}  \left|\left| \frac{d\nu}{d m^{cu}} \right|\right|_{L^{2}}\\
\nonumber &\leq m^{cu}(\pi_{i}B(\mu))  K
\end{align}

From the fact that $B(\mu)$ is $\Fcal^{ss}$-saturated and $\Fcal^{ss}$ is absolutely continuous, it follows that $m(B(\mu))>0$, so $\mu$ is a physical measure. Moreover, considering a constant that bounds from below the Jacobian of $h^{ss}$ we have that $m(B(\mu))> (2s_{0}K)^{-1} \Jac(h^{ss})^{-1} =: M$.
\end{proof}

\subsection{Proof of Theorem A}

Proposition 5.2 guarantees that every ergodic u-Gibbs is a physical measure. To conclude the proof of Theorem A, it remains to prove that there exist at most finite physical measures (Finiteness) and that the union of their basin has full Lebesgue measure in the basin of attraction (Problem of the Basins). 

\subsubsection{Finiteness}

We will show that there exist at most finite ergodic physical measures for $f$. The finiteness of physical measures will be obtained as a consequence of the full Lebesgue measure of the union of the basins of the ergodic physical measures.

Suppose that there are infinitely many ergodic u-Gibbs $\mu_{n}$, taking a subsequence we can suppose that $\mu_{n}\rightarrow \mu$, and consider $i_{0}$ and a subsequence also denoted by $\mu_{n}$ such that $\mu_{n}(C_{i_{0}})>\frac{1}{2s_{0}}$ for every $n\in\N$. Proposition 5.2 implies that every $(\pi_{i_{0}})_{*}\mu_{n}$ is absolutely continuous with respect to $m^{cu}$ and that $\frac{d(\pi_{i_{0}})_{*}\mu_{n}}{d m^{cu}}$ is uniformly bounded in $L^{2}$ by some constant $\hat{K}>0$.

Considering $\hat{M}$ as in Proposition 5.2, there exist at most $\frac{1}{\hat{M}}$ ergodic physical measures with $\left|\left|\frac{d\nu}{d m^{cu}}\right|\right|_{L^{2}}\leq \hat{K}$. Actually, supposing that there are $\mu_{1}$, $\cdots$, $\mu_{l}$ ergodic physical measures for $l\geq\frac{1}{\hat{M}}+1$, Proposition 2.1 implies that they are all ergodic u-Gibbs and Proposition 5.2 implies that the Lebesgue measure of their basin is bounded below by $\hat{M}>0$. Since their basins are disjoint, $1\geq m(\underset{i}{\cup}{B(\mu_{i}}))=\sum_{i}{m(B(\mu_{i}))}\geq l\cdot \hat{C}_{5} > 1$, which is a contradiction. Thus, there exists a finite number of ergodic physical measures.

\subsubsection{The Problem of the Basins}

We are now ready to show that the union of the basins of the physical measures have full Lebesgue measure in the whole basin of attraction.

Denote the ergodic u-Gibbs by $\mu_{1}$, $\cdots$, $\mu_{l}$, consider the sets $X=\cup B(\mu_{i})$, $E=B(\Lambda)\backslash\cup B(\mu_{i})$ and assume by contradiction that $m(E)>0$. Let $m_{E}$ be the normalized measure, $\sigma_{n}=\frac{1}{n}\sum_{j=0}^{n-1}{f^{j}_{*} (m_{E})}$ and $\sigma_{\infty}$ an accumulation point of $\sigma_{n}$. By Proposition 2.2, $\sigma_{\infty}$ is a u-Gibbs, so it projects by $\pi_{j}$ into an absolutely continuous measure with $||\frac{d(\pi_{j})_{*}\sigma_{\infty}}{dm^{cu}}||_{L^{2}}\leq K$ for every $j=1,\cdots, s_0$. Consider $\{n_k\}$ the sequence such that $\sigma_{n_k}\rightarrow \sigma_\infty$, take $j$ such that $(\pi_{j})_{*}\sigma_{\infty}$ is non-zero and define the measures $\tilde{\sigma}_{n_{k}}=(\pi_{j})_{*}\sigma_{n_{k}}$ and $\tilde{\sigma}_{\infty}=(\pi_{j})_{*}\sigma_{\infty}$.

By the invariance of $E$ and X, we have that $\sigma_{n_{k}}(X)=0$ and $0=\tilde{\sigma}_{n_{k}}(\pi_{j}(X))=\int_{\pi_{j}(X)}{\frac{d\tilde{\sigma}_{n_{k}}}{dm^{cu}} dm^{cu}}$, which implies that 
$\frac{d\tilde{\sigma}_{n_{k}}}{dm^{cu}}(x)=0$ for $m^{cu}$-almost every $x\in \pi_{j}(X)$. Analogously, $\sigma_{\infty}(E)=0$ implies that $\frac{d\tilde{\sigma}_{\infty}}{dm^{cu}}(x)=0$ for $m^{cu}$-almost every $x\in \pi_{j}(E)$. Putting it together, it holds that $\frac{d\tilde{\sigma}_{n_{k}}}{dm^{cu}}(x)\cdot \frac{d\tilde{\sigma}_{\infty}}{dm^{cu}}(x)=0$ for $m^{cu}$-almost every $x\in \tilde{W}_{j}$. Hence:

$$
\langle \frac{d\tilde{\sigma}_{n_{k}}}{dm^{cu}} , \frac{d\tilde{\sigma}_{\infty}}{dm^{cu}} \rangle_{L^{2}} = \int{\frac{d\tilde{\sigma}_{n_{k}}}{dm^{cu}}(x)\cdot \frac{d\tilde{\sigma}_{\infty}}{dm^{cu}}(x) dm^{cu}} =0,   
$$
which is a contradiction with the fact that $\langle \frac{d\tilde{\sigma}_{n_{k}}}{dm^{cu}} , \frac{d\tilde{\sigma}_{\infty}}{dm^{cu}} \rangle_{L^{2}} \rightarrow ||  \frac{\tilde{\sigma}_{\infty}}{dm^{cu}}||_{L^{2}} >0$.

\qed

\subsection{Proof of Corollary B} We will see how Corollary B follows from Theorem A  under the condition (H3) of regularity of the stable foliation.

In our setting, we have the trapping region $U$ such that the attractor $\Lambda_{f_0}$ is equals to $\displaystyle \cap_{n \geq 0 }{f_0^n (U)}$, the stable subbundle $E^{ss}_x(g)$ defined for every $x\in U$  tends to $E^{ss}_x(f_0)$ when $f$ tends to $f_0$.

\begin{prop}
Suppose that $\Fcal^{ss}$ is of class $C^{1}$ and that there are constants $a$, $\epsilon_{1}$, $\theta_{2}$, $L$ and $r_{2}$ such that: for every unstable curves $\gamma_{1}$, $\gamma_{2}$ of lengths at most $L$ and every center-unstable manifold $W^{cu}_{3}$ with $d^{ss}(\gamma_{i} , W^{cu}_{3}) \leq \epsilon_{1}$ and $d^{ss}(\gamma_{1},\gamma_{2}) \in J=[a,Ia]$, where $I= \underset{x\in B(\Lambda)}{\max} \{||Df_{|_{E^{ss}_{x}}}||^{-1}\}$, the curves $\pi^{ss}\gamma_{1}$ and $\pi^{ss}\gamma_{2}$ are $\theta_{2}$-transversal in neighborhood of radius $r_{2}$ in $W^{cu}_{3}$.

Then the Transversality Condition (H1) holds.
\end{prop}
\begin{proof}

First, let us suppose that $W^{cu}_{3}$ contains the curve $\gamma_{2}$. Given the unstable curves $\gamma_{1}$ and $\gamma_{2}$ with $d^{ss}=d<\epsilon_{0}$, we consider $\epsilon_{0}=\min\{\epsilon_{1},a \}$ and the smallest integer $n\in\N$ such that $d(f_0^{-n}\gamma_{1} , f_0^{-n}\gamma_{2}) \in J=[a,Ia]$ (this is a fundamental domain for the of stable distances of iterations of curves). By hypothesis, the projections of these curves into $f_0^{-n}W^{cu}_{2}$ are $\theta_{2}$-transversal in neighborhoods of radius $r_{2}$. Iterating forward $n$ times the projection of the curves, there exists a constant $\tau$ depending on the norm of $Df_0$ restricted to each sub-bundle such that the angle of the images of the tangent vectors to the projected curves is bounded by $\tau^{n}\theta_{2}$ in neighborhoods of radius $(\lambda_{c}^{-})^{n} r_{2}$.

\begin{claim}\label{Claim}
Given two open sets $U,V\subset \R^{2}$ with bounded diameter and a diffeomorphism $h:U\rightarrow V$ of class $C^{1}$ with $||h-id||_{C^{1}}\leq\frac{1}{2}$, for every $\theta < \frac{\pi}{3}$ and $r>0$, there exist constants $\tilde{\theta}$ and $\tilde{r}$ such that for every pair of curves $\gamma_{1}, \gamma_{2}$ contained in $U$ that are $\theta$-transversal in neighborhoods of radius $r$, the curves $\tilde{\gamma}_{1} = h(\gamma_{1})$ and $\tilde{\gamma}_{2} = h(\gamma_{2})$ are $\tilde{\theta}$-transversal in neighborhoods of radius $\tilde{r}$.
\end{claim}

\begin{proof}[Proof of Claim 5.1]
Let $v_i$ be a unit vector tangent to $\gamma_{i}$ at the point $x_{i}$ and $\tilde{v}_{i}$ a unit vector tangent to $h(\gamma_{i})$ at the point $h(x_{i})$, for $i=1,2$. The unit vector tangent to the curves $\tilde{\gamma}_{i}$ at the point $h(x_{i})\in\tilde{\gamma}_{i}$ are given by $\tilde{v}_{i}=\lambda_{i}\cdot dh_{x_{i}}(v_{i})$, where $\lambda_{i}=||dh_{x_{i}} (v_{i})||^{-1}$. 

A geometric consideration 
gives that $\sin (\frac{\angle(v,w)}{2} ) = \frac{||v-w||}{2}$ for every pair of unit vectors $v$ and $w$. Let $C$ be such that $C^{-1}\leq \frac{\sin x}{x} \leq C$ for every $x\in[\frac{-\pi}{3},\frac{\pi}{3}]$, then $C^{-1} ||v-w|| \leq \angle(v,w) \leq C ||v-w||$ whenever $\angle(v,w) < \frac{\pi}{3}$.

By continuity of the derivative, given $r$ and $\theta$ we take a constant $\tilde{r}_{1}$ such that $d(x, y)\leq \tilde{r}_{1}$ implies that $||id - dh_{x}^{-1}dh_{y} || < \frac{C^{-1}\theta}{20} $. Also note that $||h-id||_{C^{1}}\leq\frac{1}{2}$ imples that  $\frac{1}{2} \leq \lambda_1, \lambda_2 \leq 2$ and that $\frac{1}{2} \leq ||(dh_x)^{-1}||^{-1} \leq 2$.

Now, supposing that $\angle(v_1-v_2) \geq \theta$, that $v_1$ is closer to $v_2$ than to $-v_2$, and that $\hat{x}_1 = h(x_1) $, $ \hat{x}_2 = h(x_2)$ satisfies $d(\hat{x}_1 , \hat{x}_2) < \tilde{r}= \min \{\frac{r}{2}, \frac{\tilde{r}_1}{2} \}$, let us conclude that $\angle( \tilde{v}_{1}, \tilde{v}_{2})\geq \tilde{\theta}$. In fact,
\begin{align*}
||\tilde{v}_{1} - \tilde{v}_{2} || &=   ||\lambda_1dh_{x_{1}}v_{1} - \lambda_2 dh_{x_{2}}v_{2} || \\
 &\geq \lambda_{1} \cdot||dh_{x_{1}}^{-1}||^{-1} ||v_{1}- \lambda_{1}\lambda_{2}^{-1}  (dh_{x_{1}}^{-1}dh_{x_{2}}) v_{2} ||\\
 &\geq \frac{1}{4} \Big[ ||{v_{1} -  
\lambda_{1}\lambda_{2}^{-1}\cdot  dh_{x}^{-1}dh_{y}  (v_{2}))|| \Big]} \\ 
 &\geq \frac{1}{4} \Big[ ||v_{1} - \lambda_{1}\lambda_{2}^{-1}v_{2}|| - ||\lambda_{1}\lambda_{2}^{-1} (v_{2}- dh_{x}^{-1}dh_{y}  (v_{2}))|| \Big]  \\
 &\geq \frac{1}{4} \Big[ ||v_{1} - (v_{1}\cdot v_{2})v_{2}||  - \lambda_{1}\lambda_{2}^{-1} ||v_{2}-   dh_{x_1}^{-1}dh_{x_2}   (v_{2})|| \Big] \quad\quad\quad
 \end{align*}

 \begin{align*}
 &\geq \frac{1}{4} \Big[ ||v_{1} - v_{2}|| - ||v_{2} - (v_{1}\cdot v_{2})  v_{2}  || - 4 ||id - dh_{x_1}^{-1}dh_{x_2} ||   \Big]    \\
 &\geq \frac{1}{4} \Big[ ||v_{1} - v_{2}|| - | 1 - v_{1}\cdot v_{2}|    - 4 \frac{C^{-1}\theta}{20}  \Big]      \\
 &\geq \frac{1}{4} \Big[ ||v_{1} - v_{2}|| - \frac{|| v_{1} - v_{2}||^{2}}{2}    -  4\frac{C^{-1}\theta}{20}    \Big]   \\
 &\geq \frac{1}{4} ||v_{1} - v_{2}|| \Big(1 - \frac{|| v_{1} - v_{2}||}{2}\Big)    -  \frac{C^{-1}\theta}{20}       \\
 &\geq \frac{1}{4}\cdot C^{-1}\theta \cdot     \frac{1}{4} -  \frac{C^{-1}\theta}{20}       
 =\frac{C^{-1}\theta}{80}
\end{align*}

Here we have used the triangle inequality, that $||v_1 - v_2|| \leq \sqrt{2}<\frac{3}{2}  $, $||x_1-x_2|| \leq \tilde{r}_1$ and that $||v-\lambda u||$ is minimum for $\lambda = \frac{u\cdot v}{||u||^2}$. Taking $\tilde{\theta} := \frac{C^{-2}\theta}{80} $, it follows that the curves $\tilde{\gamma}_{1}$ and $\tilde{\gamma}_{2}$ are $\tilde{\theta}$-transversal in neighborhoods of radius $\tilde{r}$. 
\end{proof}

When $\gamma_{2}$ is not contained in $W^{cu}_{3}$, consider $W^{cu}_{2}$ that contains $\gamma_{2}$, the same $n$ as before $\theta=\tau^{n}\theta_{2}$ and $r= (\lambda_{c}^{-})^{n} r_{2}$. Take $\tilde{\epsilon}_{0}$ small such that it is possible to apply Claim \ref{Claim} for the stable holonomy between center-unstable manifolds with diameter smaller than $L$ whose $d^{ss}$-distance is smaller than $\tilde{\epsilon}_{0}$ (we take a local chart with product structure of $W^{cu}\times W^{ss}$ to say that the holonomy is $C^{1}$ close to the vertical projection). Taking $\epsilon_{0}<\tilde{\epsilon}_{0}$ and applying Claim 5.1 we obtain $\theta(\epsilon)$ and $r(\epsilon)$ such that the projections into $W^{cu}_{3}$ are $\theta(\epsilon)$-transversal in neighborhoods of radius $r(\epsilon)$.

\end{proof}

Now we are able to prove the robustness of the transversality condition
 under the condition of regularity of the stable foliation, which implies that the stable subbundle $E^{ss}_x(f)$ converges to $E^{ss}_x(f_0)$ for every $x\in U$ when $f$ tends to $f_0$.

\begin{prop}
The Transversality Condition (H1) is an open property under Condition (H3).
\end{prop}
\begin{proof}
As seen in Proposition 5.3, it is enough to check the transversality for projection of curves when $d^{ss}\in[a,Ia]=J$, where $J$ is a fundamental domain for the size of iterates of stable segments.

Since $x\rightarrow E^{uu}_{x}$ is continuous, we can consider a family of unstable cones with small width and fix a constant $\alpha>0$ that bounds from below the angle of each pair of  stable projections of unstable cones of this family when $d^{ss}\in J$. If a curve at $x$ is contained in a cone $C$, then every curve at $x$ that is $C^{1}$ close is also contained in the cone $C$. Thus the family of cones and the limitation $\alpha$ of the projections can be taken constant for every projection of $f$ in a neighborhood of $f_0$ in the $C^{1}$ topology. \end{proof}

\begin{proof}[Proof of Corollary B]
Given $f_{0}$, there exists a neighborhood $\mathcal{U}$ of $f_{0}$ such that every $f\in\mathcal{U}$ satisfies the hypotheses of Theorem A. Actually, Conditions (H2), (H3) and robust dynamical coherence are open in $f$, and it follows from Proposition 5.4 that the Transversality Condition is open under Condition (H3).
\end{proof}

\section{Attractors with Transversality}\label{s.attractors}

We will describe how to construct a family of nonhyperbolic attractors with central direction neutral and transversality between unstable leaves via the  stable projection.

\subsection{The Attractor $F_{0}$}

Consider $M= S^{1}\times [-1,1]\times [-1,1]$ and $F_{0}: M \rightarrow M$ of the type 
$$F_{0}(x,y,z) = \left(l x, \lambda_{c} y + g(x) , \lambda_{ss} z + h(x)\right)$$
where $\lambda_{ss}<\lambda_{c}<1$, $\lambda_{c}>\frac{1}{l} $ and $\lambda_{uu}=l\in\N$.


In what follows, we will describe two examples of such attractors $F_0$ that satisfies the transversality condition.

\subsubsection{Example 1:}
Let $l=\lambda_{uu}=3$ and $i\in\{1,2,3\}$. We consider the rectangles $R_{1}=[0,\frac{1}{3})\times [-1,1]^{2}$, $R_{2}=[\frac{1}{3},\frac{2}{3})\times [-1,1]^{2}$ and $R_{3}=[\frac{2}{3},1)\times [-1,1]^{2}$ and the functions $g$ and $h$ such that $g(x)=\alpha_{i} x + c_i$ and $h(x)=d_i$ if $x\in R_{i}$, where the constants $c_{i}$'s and $d_{i}$'s are parameters of translation guaranteeing that $F_{0}(M)\subset M$ and that $F_0$ is invertible.

The dynamics $F_0$ is a linear model of a hyperbolic attractor, restricted to each rectangle $R_{i}$ it is a hyperbolic affine transformation inserting $R_{i}$ into $M$ with slope $\alpha_{i}$ in the $y$-direction into the $xy$-plane. For simplicity we will consider a constant $\alpha\in(0,1-\lambda_{c})$ and choose the parameters $\alpha_{1}=\alpha$, $\alpha_{2}=0$, $\alpha_{3}=-\alpha$ and $c_{i}=-\alpha_{i}$, $i=1,2,3$, then the rectangles are inserted transversally with respect to the stable (vertical) foliation. This dynamics can be represented by the picture below.

\begin{figure}[h!]
  \centering
  \includegraphics[width=0.8\textwidth]{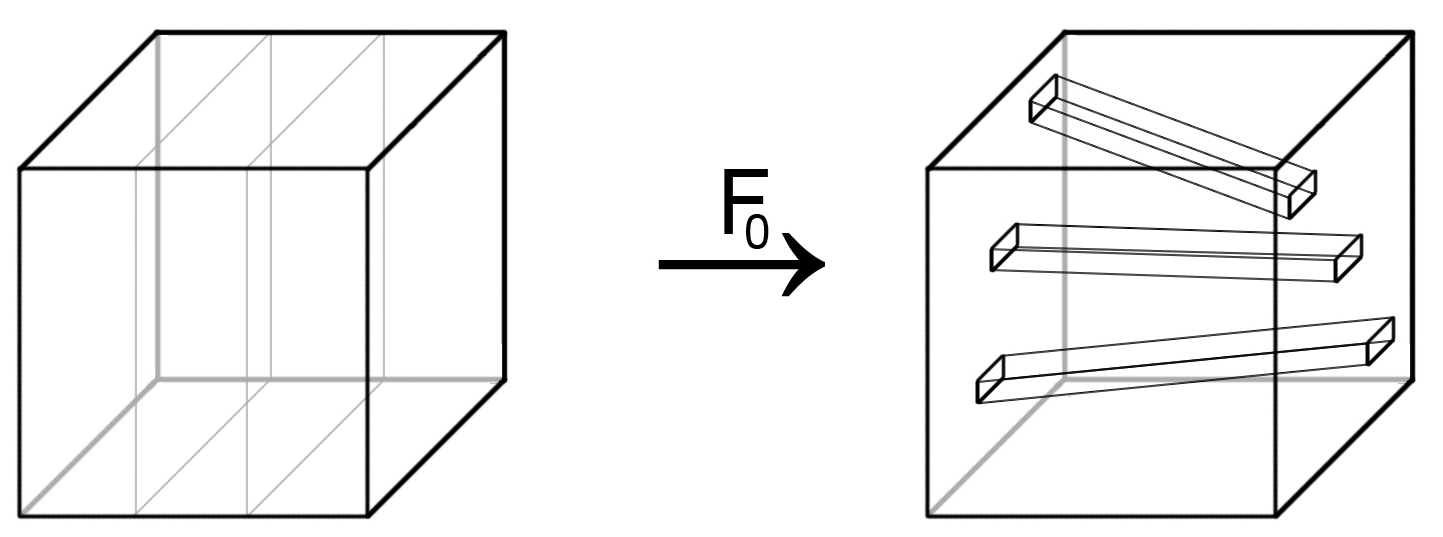}\\
\end{figure}

The sets $D_i= \{\frac{i}{3}\} \times [-1,1]^2$ are sent by $F_i$ into $\{0\}\times [-1,1]^2$. $F_0$ is uniformly hyperbolic when restricted to the interior of the $R_i$'s, so the unstable manifolds are well defined for every point in the attractor whose backward orbit never intersects any $D_i$, which corresponds to almost every point.

\begin{prop}
The attractor $\Lambda_{0}$ for $F_{0}$ satisfy the Transversality Condition (H1).
\end{prop}

\begin{proof}
The calculation is done noticing that the stable projection $\pi^{ss}$ of $F_{0}$ coincides with the vertical projection, the center-unstable direction at every point corresponds to the $xy$-plane and $E^{uu}$ is given for every $p\in\Lambda_0$ by $E^{uu}(p) = \big(1, \sum_{j}\lambda_{uu}^{-1} {\alpha_{j}(p)} \big(\frac{\lambda_{c}}{\lambda_{uu}}\big)^{j}, 0 \big)$, where $\alpha_{j}(p) = \alpha(F_0^{-j-1}(p) ) \in \{ \alpha_{1}, \alpha_{2}, \alpha_{3} \}$ depends on the itinerary of the $x$-coordinate of $p$ by the expansion in $S^{1}$ by the factor $\lambda_{uu}=3$. We obtain the expression for $E^{uu}_{p} = (1,\alpha^{uu}(p),0)$ rewriting  the condition $DF_{0}(p)\cdot E^{uu}_{p}=E^{uu}_{F_{0}(p)}$ as $(\lambda_{uu},\alpha_{i}(p)+\alpha^{uu}(p)\lambda_{c},0) = \lambda_{uu}(1, \alpha^{uu}(F_{0}(p)),0)$, so $\alpha^{uu}$ is a function that satisfies $\alpha^{uu}(F_{0}(p))=(\lambda_{uu})^{-1}\alpha(p) + \frac{\lambda_{c}}{\lambda_{uu}} \alpha^{uu}(p)$, where $\alpha(x,y,z)=\alpha_i x$. Iterating the last equation $j$ times and considering $\rho:=\frac{\lambda_{c}}{\lambda_{uu}}<\frac{1}{3}$, we get:
\begin{align*}
\alpha^{uu}(F_{0}^{j}(p))=(\lambda_{uu})^{-1}\alpha(F_{0}^{j-1}(p)) + \rho(\lambda_{uu})^{-1}\alpha(F_{0}^{j-2}(p)) + \cdots + \rho^{j}\alpha^{uu}(p)
\end{align*}

So, for every $\tilde{p}\in\Lambda$, we have:
\begin{align*}
\alpha^{uu}(\tilde{p}) &= (\lambda_{uu})^{-1}\alpha(F_{0}^{-1}(\tilde{p})) + \rho(\lambda_{uu})^{-1}\alpha(F_{0}^{-2}(\tilde{p})) + \cdots + \rho^{j}\alpha^{uu}(F_{0}^{-j}(\tilde{p}))\\
&= \sum_{j=0}^{j_{0}}{(\lambda_{uu})^{-1}\rho^{j} \alpha(F_{0}^{-(j+1)}(\tilde{p}))} + \rho^{j_{0}}\alpha^{uu}(F_{0}^{-j_{0}}(\tilde{p})) \\
&\rightarrow \sum_{j\geq 0}{(\lambda_{uu})^{-1}\rho^{j}\alpha(F_{0}^{-(j+1)}(\tilde{p}))}
\end{align*}

Consider multi-indexes $[k]=(k_{0},\cdots,k_{t-1})$ of size $t$, where $k_{j}\in\{1,2,3\}$, define the rectangle $R_{[k]}$ as the set of points $q$ such that $F_0^{j}(q)\in R_{k_{j}}$, $0\leq j \leq t-1$ and $f_{[k]}$ as the restriction of $F_0^{t}$ to $R_{[k]}$. Since the $d^{ss}$-diameter of $R_{(k_{0},\cdots,k_{t-1})}$ is equals to $(\lambda_{ss})^{t}$, if $d^{ss}(\gamma_{1}^{uu},\gamma_{2}^{uu})>\epsilon$ then these curves are contained in distinct $f_{[k]}(M)$, $f_{[k']}(M)$ for multi-indexes $[k]=(k_{0},\cdots,k_{t_{0}-1})$ and $[k']=(k'_{0},\cdots,k'_{t_{0}-1})$ of size $t_{0}\geq \frac{\log\epsilon}{\log\lambda_{ss}}$. For $p\in f_{[k]}(M)$ and $p'\in f_{[k']}(M)$, taking $j_{0}=\inf\{j, k_{j}\neq k'_{j}\}\leq t_{0}$ we have:
\begin{align*}
\alpha^{uu}(p)-\alpha^{uu}(p') &= \sum_{j\geq j_{0}}{\frac{1}{\lambda_{uu}}\rho^{j} \left[ \alpha(F^{-(j+1)}(p)) - \alpha(F^{-(j+1)}(p'))\right]}\\
&= \frac{\rho^{j_{0}}}{\lambda_{uu}}\left[\alpha(F^{-(j_{0}+1)}(p)) - \alpha(F^{-(j+1)}(p'))\right] + \sum_{j>j_{0}}{\frac{\rho^{j}}{\lambda_{uu}} (\alpha - \alpha')}\\
&\geq \frac{1}{\lambda_{uu}}\left[\rho^{j_{0}}\alpha - 2\alpha\cdot \frac{\rho^{j_{0}+1}}{1-\rho}\right]\\
&\geq \frac{1}{\lambda_{uu}}\rho^{\frac{\log\epsilon}{\log\lambda_{ss}}}\alpha\left(\frac{1-3\rho}{1-\rho}\right) := C(\epsilon)
\end{align*}

Since $\pi^{ss}_{0}E^{uu}(p) = (1,\alpha^{uu}(p))$, it follows that $\angle( (1,\alpha^{uu}(p)),(1,\alpha^{uu}(p') )$ is bounded from below, because the function $\alpha^{uu}$ is uniformly bounded and the difference $|\alpha^{uu}(p)-\alpha^{uu}(p')|$ is bounded from below by $C(\epsilon)$. Then there exists a function $\theta:(0,1)\rightarrow\R^{+}$ such that the angle between $\pi^{ss}(\gamma_{1}^{uu})$ and $\pi^{ss}(\gamma_{2}^{uu})$ is greater than $\theta(\epsilon)$ whenever $d^{ss}(\gamma_{1}^{uu} , \gamma_{2}^{uu})>\epsilon$.
\end{proof}

\subsubsection{Example 2:}Consider
 $F_0:M\rightarrow M$ given by
$$F_{0}(x,y,z) = \left(2x, \lambda_{c} y + \sin(2\pi x), \lambda_{ss} z + h(x) \right)$$
with   $\lambda_{ss}<0,5 <\lambda_{c}< 0,51$, $l=\lambda_{uu}=2$, $h(x)$ equals to $\frac{1}{2}$ if $0\leq x\leq \frac{1}{2}$ and equals to  $-\frac{1}{2}$ if $\frac{1}{2}\leq x \leq 1$.
This dynamics is an invertible version of the example 1 in \cite{fat}, considering
$M$ the quotient of $[0,1]\times[-1,1]\times[-1,1]$ by the identification  $(1,y,z)\sim (0,y,1-z)$.

\begin{prop}
 The attractor $\Lambda_{0}$ for $F_{0}$ satisfy the Transversality Condition (H1).
\end{prop}

\begin{proof}
Once again the stable projection $\pi^{ss}$ is the projection into the $z$-coordinate, the center-unstable direction is tangent to  the $xy$-plane and $E^{uu}$ is given for every $p\in\Lambda_0$ by $E^{uu}(p) = \big(1, \sum_{j}{ \lambda_{uu}^{-1} \alpha(F_0^{-j-1}p)} \big(\frac{\lambda_{c}}{\lambda_{uu}}\big)^{j}, 0 \big)$, where $\alpha(x,y,z)=\frac{\partial
}{\partial x}\sin(2\pi x)=2\pi \cos(2\pi x)$.

From Proposition 5.3, it is sufficient to check the transversality for every pair of unstable curves whose stable distance is in the interval $J=[\frac{1}{4},\frac{1}{4 \lambda_{ss} }]$. This is obtained considering $a=\frac{1}{4}$, $I=\frac{1}{\lambda_{ss}}>2$ and $\epsilon_1=\frac{1}{4 \lambda_{ss}}$ in the statement of Proposition 5.3.

 Given two unstable curves $\gamma_1^{uu}$ and $\gamma_2^{uu}$ with $d^{ss}(\gamma_1^{uu} ,\gamma_2^{uu} )\in J$, $p=(x,y,z) \in \gamma_1^{uu}$, $p'=(x',y',z')\in  \gamma_2^{uu}$, the condition $d^{ss}(\gamma_{1}^{uu},\gamma_{2}^{uu}) \in J$ implies that $\gamma_{1}^{uu}$ and $\gamma_{2}^{uu}$ are not contained in the same set $F_0(R_1)$ or $F_0(R_2)$.

 Denoting $F_0^{-j}(p)=(x_j,y_j,z_j)$ and $F_0^{-j}(p')=(x'_j,y'_j,z'_j)$. When $x_1,x'_1 \in [0,\frac{2}{5}]\cup [\frac{3}{5},1]$ one can check that   $\Big| \cos(2\pi x_1) - \cos(2\pi x'_1) \Big| \geq \cos(\frac{2\pi}{5})$ and $\Big|  \cos(2\pi x_2) - \cos(2\pi x'_2)) \Big| \leq \cos(\frac{\pi}{5})+\cos(\frac{\pi}{2}-\frac{\pi}{5}) $. This implies that

\begin{align*}
\Big|\alpha^{uu}(p)-\alpha^{uu}(p') \Big| &= \Big|\sum_{j\geq 1}{\frac{1}{\lambda_{uu}}\rho^{j} (\alpha(F^{-(j+1)}(p))- \alpha(F^{-(j+1)}(p'))}) \Big|  \\
&\geq \pi \Big| \cos(2\pi x_1) - \cos(2\pi x'_1) \Big|
 - \frac{\pi\lambda_c}{2} \Big|  \cos(2\pi x_2) - \cos(2\pi x'_2)) \Big|\\
 &\quad\quad\quad -  \Big| \sum_{j>2}{\frac{2\pi \rho^{j}}{\lambda_{uu}} ( \alpha(x_j) - \alpha(x'_j)}   \Big|   
\end{align*}
which is bounded from below by $K=2\pi\Big[\cos(\frac{2\pi}{5}) - \frac{\lambda}{4}\Big(\cos(\frac{\pi}{5})+\cos(\frac{\pi}{2}-\frac{\pi}{5})  \Big) - \frac{\lambda^2}{4}\frac{2}{2-\lambda}\Big] > 0$.

  When $x_1,x'_1 \in [\frac{2}{5}, \frac{3}{5}]$, a lower bound for the distance of $\pi^{ss}(x_1)$ and $\pi^{ss}(x'_1)$ is obtained similarly using the relation $y= \sum_{i=1}^{\infty} \lambda_c^{i-1}g(F_0^{-i}(p) )$, which is verified using that
$F_0^n(p)=\big(\lambda_{uu}^n x,\lambda_c^n y + \sum_{j=1}^n \lambda_c^{n-j} g(F_0^{j-1}(p) ), \lambda_{ss}^n z + \sum_{j=1}^n \lambda_{ss}^{n-j} g(F_0^{j-1}(p)) \big)$,
denoting $g(x,y,z)=g(x)$ and $h(x,y,z)=h(x)$
,
taking $p=F_0^n(F_0^{-n}(p))$
and sending $n\rightarrow\infty$.
 The condition $x_1,x'_1 \in [\frac{2}{5}, \frac{3}{5}]$ implies that the distance of the points $x_2$, $x_2'$ to the set $\{\frac{1}{8} , \frac{3}{8} , \frac{5}{8} , \frac{7}{8}\}$ is smaller than $\frac{1}{40}$, implying 
$$|\pi^{ss}(p) - \pi^{ss}(p') | \geq | y_1 - y_1' | \geq  2 \Big[ \sin\Big(\frac{2\pi}{5}\Big) - \lambda_c \sin\Big(\frac{3\pi}{10}\Big) - \frac{\lambda_c^2}{1-\lambda_c} \Big] =: K_2 > 0$$
   By continuity, we have that $|\pi^{ss}(p) - \pi^{ss}(p') | > \frac{K_2}{2}$ whenever  $x_1,x'_1 \notin [\frac{2}{5}-\epsilon_1 , \frac{3}{5} +\epsilon_1]$. So, for $r_2=\min\{\frac{\epsilon_1}{2},\frac{K_2}{2}\}$, we get  $\theta_2$-transversality in neighborhoods of radius $r_2$.
\end{proof}

\subsection{The Family of Attractors $F_{\mu,n}$} Let us  give an example of nonhyperbolic attractors satisfying the transversality condition.

Let $F_{0}$ be as in any of the two examples before, we have 
$F_{0}^{n}(x,y,z)=(\tau^{n}(x), \lambda_{c}^{n}y + g_{n}(x), \lambda_{ss}^{n}z + h_n(x) ) $, where $\tau:S^{1}\rightarrow S^{1}$ is the expansion in the circle by the factor $\lambda_{uu}$,  $g_{n}(x)= \sum_{j=0}^{n-1}{\lambda_{c}^{n-j-1} g(\tau^{j}(x))}$  and $h_{n}(x)= \sum_{j=0}^{n-1}{\lambda_{c}^{n-j-1}h(\tau^{j}(x))}$.

Let $q$ be a periodic of $F_{0}$, $\delta<\frac{1}{10\lambda_{uu}}$, an affine function $\psi_{0}:B^{cu}(p,\delta)\subset M \rightarrow B(0,\delta)\subset \R^{2}$ that sends $q\rightarrow 0$ and $\{E^{uu}_{0},E^{c}_{0}\}\rightarrow \{e_{1},e_{2}\}$, take a bump function $\psi_{1}: B(0,\delta)\subset \R^{2} \rightarrow [0,1]$ of class $C^{\infty}$, with
$\psi_{1}(x)=1$ if $||x||\leq\frac{\delta}{3}$, $\psi_{1}(x)=\in (0,1)$ if $\frac{1}{3}\leq||x||\leq\frac{2}{3}$ and $\psi_{1}(x)=0$ if $||x||\geq\frac{2\delta}{3}$, this is done with $||\psi_{1}||_{C^{1}} \leq 2\frac{\delta^{-1}}{3}$, $||\psi_{1}||_{C^{0}}\leq 1$ and $||\psi_{0}||_{C^{0}}\leq \delta$. Fixing some $\lambda_{c}^{+}>1$,  define $\Phi_{\mu,n}:S^{1}\times[-1,1]^{2}\rightarrow \R$ by $
\Phi_{\mu,n}(x,y,z)= \mu \psi_{1}\left(\psi_{0}(x,y)\right) \Big[ (\lambda_{c}^{+} - \lambda_{c}^{n})y \Big]$ and define the family $F_{\mu,n}: M \rightarrow M $ by
$$F_{\mu,n}(x,y,z) = \Big( \lambda_{uu}^{n}, \lambda_{c}^{n} y + g_{n}(x)  + \Phi_{\mu,n}(x,y) , \lambda_{ss}^{n} z + h_n(x) \Big)$$

This family $F_{\mu,n}$ corresponds to a deformation of $F_{0}^{n}$ changing the index of the  periodic point $p$ when passing through a pitchfork bifurcation. The deformation is done along the central direction, keeping the same central direction for every parameter $\mu$. The attractor is the set $\Lambda_{\mu,n}=\cap_{j\geq 0} F_{\mu,n}^{j}(M)$. We will see that for an appropriate choice of $n$, it is possible to keep close the unstable direction in order that the transversality condition still holds for every parameter $\mu\in[0,1]$.


\begin{prop}
For every $\epsilon_{1}>0$, there exists an integer $n_{0}\in\N$ such that  $d(E^{uu}_{F_{\mu,n}}(x) , E^{uu}_{F_{0}^{n}}(x))<\epsilon_{1}$ for every $x\in\Lambda_{\mu,n}$, every $\mu\in [0,1]$ and every $n\geq n_{0}$.
\end{prop}

\begin{proof}
For parameters $\mu$ and $n$, the unstable direction $E^{uu}_{\mu}(q)=(1,\alpha^{uu}_{\mu}(q),0)$ satisfies $DF_{\mu,n}(q)\cdot(1,\alpha^{uu}(q),0)=\lambda_{uu}^n(1,\alpha^{uu}(F_{\mu,n}(q)),0)$ for every $q\in \Lambda_{\mu,n}$, which can be rewritten as:
\begin{align*}
\alpha^{uu}(F_{\mu,n}(q))=\lambda_{uu}^{-n}g_{n}+\lambda_{uu}^{-n}\frac{\partial\Phi_{\mu,n}}{\partial x} + \lambda_{uu}^{-n}\Big(\lambda_{c}^{n}+\frac{\partial\Phi_{\mu,n}}{\partial y}\Big)\alpha^{uu}(q)
\end{align*}

The function $\alpha^{uu}$ is obtained as the fixed point of the operator $T_{\mu,n}$ given by
\begin{align*}
T_{\mu,n}(\alpha)(q)=
\lambda_{uu}^{-n} \frac{d g_{n}}{dx}(F_{\mu,n}^{-1}(q)) &+ \lambda_{uu}^{-n} \frac{\partial\Phi_{\mu,n}}{\partial x}(F_{\mu,n}^{-1}(q)) \\
+ &\lambda_{uu}^{-n}\Big(\lambda_{c}^{n} +\frac{\partial\Phi_{\mu,n}}{\partial y}\Big)\alpha(F_{\mu,n}^{-1}(q))
\end{align*}

This fixed point $\alpha^{uu}_{\mu,n}$ is well defined because the operator $T_{\mu,n}$ is a contraction in the Banach space of continuous functions $\beta:\Lambda\rightarrow \R$ endowed with the norm of the supremum, in fact, the relation $\lambda_{\mu,n} = \lambda_{uu}^{-n}\left(\lambda_{c}^{n}+\frac{\partial\Phi_{\mu,n}}{\partial y}\right)\leq\eta<1$ guarantees that the operator is a contraction.

It remains to check that $||\alpha^{uu}_{\mu,n} - \alpha^{uu}_{0,n}||<\epsilon_{1}$, which is valid when the respective operators are close. Writing $T_{\mu,n}(\alpha)=A_{\mu,n} + \lambda_{\mu,n}\alpha\circ F_{\mu,n}^{-1}$, we have:
\begin{align*}
d(\alpha^{uu}_{\mu,n} , \alpha^{uu}_{0,n}) &= d(T_{\mu,n}(\alpha^{uu}_{\mu,n}) , T_{0,n}( \alpha^{uu}_{0,n}) )  \\
&\leq d(T_{\mu,n}(\alpha^{uu}_{\mu,n}) , T_{\mu,n}( \alpha^{uu}_{0,n}) ) + d(T_{\mu,n}(\alpha^{uu}_{0,n}) , T_{0,n}( \alpha^{uu}_{0,n}) )  \\
&\leq || T_{\mu,n} || \cdot  d(\alpha^{uu}_{\mu,n} , \alpha^{uu}_{0,n}) + || A_{\mu,n} - A_{0,n} || + ||\lambda_{\mu,n} \alpha^{uu}_{0,n} - \lambda_{0,n} \alpha^{uu}_{0,n} ||
\end{align*}

Therefore
\begin{align*}
d(\alpha^{uu}_{\mu,n} , \alpha^{uu}_{0,n}) &\leq \frac{d(T_{\mu,n},T_{0,n})}{1 - || T_{\mu,n} ||}\\
 &\leq  \frac{ ||A_{\mu,n} - A_{0,n} || + (\lambda_{\mu,n} - \lambda_{0,n}) ||\alpha_{0,n} ||}{1 - \eta}
\end{align*}

Since $\alpha_{n}(q)$ depends only in the $x$-coordinate of $q$, and since the $x$-coordinate of $F_{\mu,n}^{-1}(p)$ and of $F_{0,n}^{-1}(q)$ are the same, we have:
\begin{align*}
(T_{\mu,n}-T_{0,n})(\beta) (q) =
&(\lambda_{uu})^{-n} \left[\frac{\partial\Phi_{\mu,n}}{\partial x}(F_{\mu,n}^{-1}(q)) \right] +\\
&(\lambda_{uu})^{-n} \left[\left(\lambda_{c}^{n}+\frac{\partial\Phi_{\mu,n}}{\partial y}\right)\beta(F_{\mu,n}^{-1}(q)) - \left(\lambda_{c}^{n}\right)\beta(F_{0,n}^{-1}(q)) \right]
\end{align*}

Each line above goes to zero when $n\rightarrow\infty$, since $\lambda_{c}^{n}(\lambda_{uu})^{-n} \overset{n\rightarrow\infty}{\rightarrow} 0$ and $\left|\frac{\partial\Phi_{\mu,n}}{\partial x}\right|$, $\left|\frac{\partial\Phi_{\mu,n}}{\partial y}\right|$ are bounded. Then we take $n$ large, $D_{1}$ and $\lambda_{1}$ small such that:
\begin{align*}
||\alpha^{uu}_{\mu,n} - \alpha^{uu}_{0,n}|| \leq \frac{D_{1}+\lambda_{1} ||\alpha^{uu}||}{1-\eta}<\epsilon_{1}.
\end{align*}
\end{proof}


The nonhyperbolicity of $\Lambda_{\mu,n}$ will follow from the fact that the attractor admits hyperbolic periodic points of different indexes and that it is robustly transitive.

\begin{prop}
There exists an integer $n_{1}$ such that the attractor $\Lambda_{n,\mu}$ of $F_{\mu,n}$ is robustly transitive for every $\mu\in[0,1]$ and every $n\geq n_{1}$.
\end{prop}

\begin{proof}[Proof of Proposition 6.3]

In our situation, the deformation is done inside a set $D$ such that the dynamics outside $D$ contracts vectors tangent  to the central direction. Hence, we can make an argument similar to the one of Ma\~n\'e \cite{Ma,BV} to prove the robust transversality, 

Since the unstable foliation for $F_{0}$ is minimal in $\Lambda_{0}$, for every $\varrho$ exists a corresponding $L>0$ such that the unstable foliation is $(\frac{L}{2},\frac{\varrho}{2})$-dense in center-stable manifolds, that is, every center-stable ball of radius $\frac{\varrho}{2}$ intersects every unstable curve $\gamma^{uu}$ of length greater than $\frac{L}{2}$.

Applying Proposition 6.2 for an $\epsilon_{1}$ sufficiently small, there exists $n_{1}$ such that the unstable foliation of $F_{\mu,n}$ is $(L,\varrho)$-dense in center-stable manifolds for every $n\geq n_{1}$ (the center-stable foliation is the same for $F_{\mu,n}$ and $F_{0}$). Let  $\mathcal{U}_{1}$ be  a neighborhood of $F_{\mu,n}$ such that the same $(L,\varrho)$-density holds for every $f$ in this open set $\mathcal{U}_{1}$. In what follows, we assume that $\delta\leq\frac{1}{10 \lambda_{uu}}\leq \frac{1}{20}$.

\begin{claim}
For every open set $U$ intersecting $\Lambda$ there exist a point $x\in U\cap\Lambda$ and an integer $N\in\N$ such that $f^{-n}(x)\notin D$ for every $n\geq N$.
\end{claim}
\begin{proof}[Proof of Claim 6.1] 
Take a negative iterate $\tilde{U}$ of $U$ with central-stable diameter greater than $100\delta$, which exists since $Df_{|_{E^{cs}}}$ contracts area. Considering $R=3\lambda_{ss}$, we note that every center-stable leaf of radius greater than $R$ contains a center-stable ball $B$ of radius $\frac{R}{3}$ such that $f^{-1}(B)$ does not intersect $D$.

Let $\Gamma_{1}$ be a subset of $\tilde{U}$ such that $f^{-1}(\Gamma_{1})\cap D=\emptyset$. Since the center-stable diameter of $f^{-1}(\Gamma_{1})$ is greater than $R$, we can consider $\Gamma_{2}\subset f^{-1}(\Gamma_{1})$ such that $f^{-1}(\Gamma_{2})\cap D=\emptyset$, inductively we consider a sequence $\Gamma_{n}\subset f^{-1}(\Gamma_{n-1})$ such that $f^{-1}(\Gamma_{n})\cap D = \emptyset$. So, any $x\in\underset{n}{\cap}{f^{n}(\Gamma_{n})}$ is a point that satisfies the conclusion of the Claim.
\end{proof}

Given two open subsets $U$ and $V$ of $\Lambda$, we consider a point $x\in U\cap\Lambda$ given by Claim 6.1 and an unstable curve contained in $V$. Iterate the unstable curve until it has length greater than $L$ and preiterate the center-stable leaf of $x$ contained in $U$ until it has internal radius greater than $2\delta$. Due to the $(L,\delta)$-density we know that these sets have non-empty intersection, so $U\cap f^{n}(V)\neq\emptyset$ for some $n\geq 0$.
\end{proof}

\subsection{Proof of Theorem C}
\begin{prop}
There exists an integer $n_{2}\in\N$ such that the Transversality Condition holds for the dynamics $F_{\mu,n}$, for every $\mu\in[0,1]$ and every $n\geq n_{2}$.
\end{prop}
\begin{proof}
As seen in Proposition 6.1, the transversality holds for every $F_{0,n}$. Fix $a\in(0,\frac{\lambda_{ss}}{10})$, a fundamental stable domain $J=[a,\lambda_{ss}^{-1}a]$ and consider $\theta(a)$ a lower bounf of the angle of projections of unstable curves with $d^{ss}\in J$ for $F_{0}$.

Since the stable projection is the vertical one, for each unstable curve we consider $\omega>0$ small such that the center-unstable cone of width $\omega$ around the unstable direction is $\frac{\theta(a)}{2}$-transversal to each other center-unstable cone of width $\omega$ around other unstable direction with stable distance between them contained in the interval $J$, by compactness this width $\omega$ can be taken uniform. Fixing this family of unstable cones of width $\omega$ containing the original $E^{uu}_{0,n}$ we have, by Proposition 6.2, that there exists $n_{2}\in\N$ such that this family of cones contains the unstable direction of $F_{\mu,n}$ for every $\mu\in[0,1]$ and every $n\geq n_{2}$. This implies that $\angle (\pi^{ss}(E^{uu}_{\mu,n}(x_{1})),\pi^{ss}(E^{uu}_{\mu,n}({x_{2}})))>\frac{\theta(a)}{2}$ whenever $d^{ss}(x_{1},x_{2})>a$. By Proposition 5.1, this is enough to guarantee the transversality.
\end{proof}

\begin{proof}[Proof of Theorem C]
Take $F_0$ as in any of the two examples above and construct $F_{\mu,n}$ as before with $\lambda_{c}^{+}>1$ such that $\frac{(\lambda_{c}^{+})^{2}}{3\lambda_{c}}<1$. We have that $\frac{(\lambda_{c}^{+})^{2}}{(3\lambda_{c})^{n}}<1$ (central direction is neutral for $F_{\mu,n}$), which guarantees Condition (H2). Taking $n_{3}\in\N$ such that the Proposition 6.3 is valid for every $F_{\mu,n}$, $\mu\in[0,1]$, $n\geq n_{3}$, it follows that the Transversality Condition holds for every $F_{\mu,n}$. The $C^{1}$ regularity follows by choosing $\lambda_{ss}<1$ with $\frac{3\lambda_{ss}}{\lambda_{c}^{+}}<1$ (this condition implies the $C^{1}$ regularity of the stable foliation and also the continuity of this foliation in the $C^{1}$-topology). The attractor is robustly dynamically coherent because the central direction for $F_{\mu,n}$ is of class $C^{1}$, actually, $E^{c}= y$-direction, $E^{cu}= xy$-plane, $E^{cs}=yz$-plane for every $F_{\mu,n}$ and Theorems 7.1 and 7.4 of \cite{HPS} guarantees that these laminations are structurally stable. Taking $N>n_{2}$ as given in Proposition 6.3, the map $F_{1,N}$ is robustly nonhyperbolic because it has periodic points of different indexes and it is robustly transitive.
\end{proof}


\begin{thebibliography}{A}

\bibitem{ABV} Alves, J.; Bonatti, C.; Viana, M.; \tb{SRB measures for partially hyperbolic systems whose central directions is mostly expanding}, \ti{Invent. Math., 140, 351-398},  (2000).

\bibitem{araujo} Araujo, V.; \tb{Attractors and time averages for random maps}, \ti{Annales Inst. Henri Poincar\'e - Analyse Non-Lin\'eaire, 17, 307-369}, (2000).

\bibitem{araujo.melbourne} Araujo, V.; Melbourne, I.; \tb{Existence and smoothness of the stable foliation for sectional hyperbolic attractors}, \ti{preprint, arXiv:1604.06924}, (2016).



\bibitem{Be.Ca} Benedicks, M.; Carleson, L.; \tb{The dynamics of the H\'enon map}, \ti{Annals of Math., 133, 73-169}, (1991).

\bibitem{Be.Vi} Benedicks, M.; Viana, M.; \tb{Solution of the basin problem for H\'enon-like attractors}, \ti{Invent. Math., 143, 375–434},  (2001).

\bibitem{Be.Yo} Benedicks, M.; Young L.-S.; \tb{Sinai-Bowen-Ruelle measures for certain H\'enon maps}, \ti{Invent. Math., 112, 541–576},  (1993).

\bibitem{BDV} Bonatti, C.; Diaz, L.; Viana, M.; \tb{Dynamics Beyond Uniform Hyperbolicity}, \ti{Encyclopaedia of Mathematical Sciences, vol. 102, Springer Verlag}, (2005).

\bibitem{BV} Bonatti, C.; Viana, M.; \tb{SRB measures for partially hyperbolic systems whose central directions is mostly contracting}, \ti{Israel J. of Math., 115, 157-193}, (2000).
 
\bibitem{Bowen} Bowen., R.; \tb{Equilibrium states and the ergodic theory of Anosov diffeomorphisms}, \ti{Lect. Notes in Math., vol. 470 , Springer Verlag}, (1975). 

\bibitem{Bowen.Ruelle} Bowen, R.; Ruelle, D.; \tb{The ergodic theory of Axiom A flows}, \ti{Invent. Math., 29, 181-202}, (1975).
 
\bibitem{Carvalho} Carvalho, M.; \tb{Sinai-Ruelle-Bowen measures for N-dimensional derived from Anosov diffeomorphisms}, \ti{Ergodic Theory and Dynamical Systems, 13, 21-44}, (1993).

\bibitem{measurability} Castaing, C.; Valadier, M.; \tb{Convex Analysis and Measurable Multifunctions}, \textit{Lect. Notes in Math., vol 580}, Springer-Verlag, (1977).

\bibitem{C} Chernov, N.; \tb{Markov approximations and decay of correlations for Anosov flows}, \ti{Annals of Math., 147, 269-324}, (1998).

\bibitem{D} Dolgopyat, D.; \tb{On Decay of correlations in Anosov flows}, \ti{Annals of Math., 147, 357-390}, (1998).
 
\bibitem{HPS} Hirsch, M.; Pugh, C.; Shub, M.; \tb{Invariant manifolds}, \textit{Lect. Notes in Math., vol 583}, Springer-Verlag, (1977).

\bibitem{HHU} Hertz, F. R.; Hertz, J. R.; Ures, R.; \tb{A Survey on Partially Hyperbolic Dynamics}, \textit{Partially hyperbolic dynamics, laminations, and Teichm¨uller flow, Fields Inst. Commun.}, vol. 51, Amer. Math. Soc., (2007).

\bibitem{L} Liverani, C.; \tb{On Contact Anosov Flows}, \ti{Annals of Math., 159, 1275-1312}, (2004).

\bibitem{lyubich} Lyubich, M.; \tb{Almost every real quadratic map is either regular or stochastic}, \ti{Annals of Math., 156, 1-78,} (2002).

\bibitem{Ma} Ma\~n\'e, R.; \tb{Contributions to the Stability Conjecture}, \ti{Topology, 17, 383-396}, (1978).


\bibitem{Palis2} Palis, J.; \tb{A Global Perspective for Non-Conservative Dynamics}, \ti{Ann. Inst. Henri Poincar\'e, Analyse Non Lineaire, Vol. 22}, (2005).

\bibitem{Pesin} Pesin, Y.; \tb{Dynamical systems with generalized hyperbolic attractors: hyperbolic, ergodic and topological properties}, \ti{Ergod. Th.  Dynam. Sys., 12, 123-151}, (1992).

\bibitem{Pe,Si} Pesin, Y.; Sinai, Y.; \tb{Gibbs measures for partially hyperbolic attractors}, \ti{Ergod. Th.  Dynam. Sys., 2, 417-438}, (1982).


\bibitem{Pujals} Pujals, E.; \tb{Density of hyperbolicity and homoclinic bifurcations for topologically hyperbolic sets}, \ti{Discr. and Cont. Dynam. Sys., 20, 337-408}, (2008).

\bibitem{Ruelle} Ruelle, D.; \tb{A measure associated with Axiom A attractors}, \ti{Amer. J. Math., 98, 619-654}, (1976).

\bibitem{Sataev} Sataev, E.; \tb{Invariant measures for hyperbolic maps with singularities}, \ti{Russ. Math. Surveys, 471, 191-251}, (1992).

\bibitem{Sinai} Sinai, Y.; \tb{Gibbs measure in ergodic theory}, \ti{Russian Math. Surveys, 27, 21-69}, (1972).

\bibitem{fat} Tsujii, M.; \tb{Fat Solenoidal Attractor}, \ti{Nonlinearity, 14, 1011-1027}, (2001).

\bibitem{Tsujii} Tsujii, M.; \tb{Physical measures for partially hyperbolic surface endomorphism}, \ti{Acta Mathematica, 194, 37-132}, (2005).

\bibitem{Tucker} Tucker, W.; \tb{The Lorenz attractor exists}, \ti{C. R. Acad. Sci. Paris, 328, Serie I, 1197-1202}, (1999).

\bibitem{VY} Viana, M.; Yang, J.; \textbf{Physical measures and absolute continuity for one-dimensional center direction}, \ti{Annales Inst. Henri Poincar\'e - Analyse Non-Lin\'eaire, 30, 845-877}, (2013).

\bibitem{ls.young} Young, L. S.; \tb{Statistical properties of dynamical systems with some hyperbolicity}, \ti{Annals of Math., 147, 585-650}, (1998).

\end{thebibliography}
\end{document}